\documentclass{article}
\usepackage[utf8]{inputenc}
\usepackage{amsthm}
\usepackage{amsmath}
\usepackage{amstext}
\usepackage{amssymb, comment}
\usepackage[all]{xy}
\usepackage{color}
\usepackage{tikz-cd}

\newtheorem{definition}{Def\text{}inition}[section]
\newtheorem{theorem}[definition]{Theorem}
\newtheorem{example}[definition]{Example}

\newtheorem{lemma}[definition]{Lemma}
\newtheorem{proposition}[definition]{Proposition}
\newtheorem{corollary}[definition]{Corollary}

\newtheorem{question}[definition]{Question}

\begin{document}

\title{ \bf \large Iterations and unions of star selection properties on topological spaces\footnote{The first-listed author was supported for this research by Consejo Nacional de Ciencia y Tecnología (CONACYT, México), Scholarship 769010.}}
\author{ \small JAVIER CASAS-DE LA ROSA, WILLIAM CHEN-MERTENS, SERGIO GARCIA-BALAN}
\date{}
\maketitle

\begin{abstract} 
In this paper, we investigate what selection principles properties are possessed by small (with respect to the bounding and dominating numbers) unions of spaces with certain (star) selection principles.. Furthermore, we give several results about iterations of these properties and weaker properties than paracompactness. In addition, we study the behaviour of these iterated properties on $\Psi$-spaces. Finally, we show that, consistently, there is a normal star-Menger space that is not strongly star-Menger; this example answers a couple of questions posed in \cite{CGS}.
\end{abstract}



\section{Introduction}

\subsection{Notation and terminology}

Let $X$ be a set and let $\mathcal{U}$ be a collection of subsets of $X$. If $A$ is a subset of $X$, then the star of $A$ with respect to $\mathcal{U}$, denoted by $St(A,\mathcal{U})$, is the set $\bigcup\{U\in\mathcal{U}:U\cap A\neq\emptyset\}$; for $A=\{x\}$ with $x\in X$, we write $St(x,\mathcal{U})$ instead of $St(\{x\},\mathcal{U})$. We denote by $[X]^{<\omega}$ the collection of all finite subsets of $X$.  Throughout this paper, all spaces are assumed to be regular, unless a specific separation axiom is indicated. For notation and terminology, we refer to \cite{E}.

We recall some classical star covering properties following the terminology of \cite{DRRT}. A space $X$ is said to be strongly starcompact (strongly star-Lindel\"{o}f), briefly $SSC$ ($SSL$), if for every open cover $\mathcal{U}$ of $X$ there exists a finite (countable) subset $F$ of $X$ such that $St(F,\mathcal{U})=X$. A space $X$ is starcompact (star-Lindel\"{o}f), briefly $SC$ ($SL$), if for every open cover $\mathcal{U}$ of $X$ there exists a finite (countable) subset $\mathcal{V}$ of $\mathcal{U}$ such that $St(\bigcup\mathcal{V},\mathcal{U})=X$. It is well-known that countable compactness and strongly starcompactness are equivalent for Hausdorff spaces (see \cite{DRRT}). We refer the reader to the survey of Matveev \cite{M} for a more detailed treatment of these star covering properties.

Recall that a space $X$ is said to be \emph{metacompact} (\emph{metaLindel\"{o}f}) if every open cover $\mathcal{U}$ of $X$ has a point-finite (point-countable) open refinement $\mathcal{V}$. Further, a space $X$ is said to be \emph{paracompact} (\emph{paraLindel\"{o}f}) if every open cover $\mathcal{U}$ of $X$ has a locally-finite (locally-countable) open refinement $\mathcal{V}$. For more information about the relationships among these covering properties (and others), we refer the reader to \cite{Burke}.\\

Recall that for $f,g\in\omega^\omega$, $f\leq^* g$ means that $f(n)\leq g(n)$ for all but finitely many $n$ (and $f\leq g$ means that $f(n)\leq g(n)$ for all $n$). A subset $B$ of $\omega^\omega$ is \emph{bounded} if there is $g\in\omega^\omega$ such that $f\leq^*g$ for each $f\in B$. A subset $D$ of $\omega^\omega$ is \emph{dominating} if for each $g\in\omega^\omega$ there is $f\in D$ such that $g\leq^*f$. The minimal cardinality of an unbounded subset of $\omega^\omega$ is denoted by $\mathfrak{b}$, and the minimal cardinality of a dominating subset of $\omega^\omega$ is denoted by $\mathfrak{d}$. The family of all meager subsets of $\mathbb{R}$ is denoted by $\mathcal{M}$ and the minimum of the cardinalities of subfamilies $\mathcal{U}\subset\mathcal{M}$ such that $\bigcup\mathcal{U=\mathbb{R}}$ is denoted by $cov(\mathcal{M})$.\\

Recall that a family $\mathcal{A}$ of infinite subsets of $\omega$ is almost disjoint (a.d., for short) if the intersection of any two distinct sets in $\mathcal{A}$ is finite. Let $\mathcal{A}$ be an a.d. family, we consider $\Psi(\mathcal{A})=\mathcal{A}\cup\omega$ with the following topology: the points of $\omega$ are isolated and a basic neighbourhood of a point $a\in\mathcal{A}$ is of the form $\{a\}\cup (a\setminus F)$, where $F$ is a finite subset of $\omega$. Then $\Psi(\mathcal{A})$ is called a $\Psi$-space (see \cite{PM}).

\subsection{Classical (star) selection principles}

Likely, among classical selection principles, the most well-known are the Menger, Rothberger and Hurewicz properties. Let us recall those notions and its star versions. Given a topological space $X$, we denote by $\mathcal{O}$ the collection of all open covers of $X$ and by $\Gamma$ the collection of all $\gamma$-covers of $X$. Recall that an open cover $\mathcal{U}$ of $X$ is a $\gamma$-cover if it is infinite and each $x\in X$ belongs to all but finitely many elements of $\mathcal{U}$. A space $X$ is Menger ($M$) if for each sequence $\{\mathcal{U}_n:n\in\omega\}$ of open covers of $X$, there is a sequence $\{\mathcal{V}_n:n\in\omega\}$ such that for each $n\in\omega$, $\mathcal{V}_n$ is a finite subset of $\mathcal{U}_n$ and $\{\bigcup\mathcal{V}_n:n\in\omega\}$ is an open cover of $X$ (see \cite{Menger}). A space $X$ is Rothberger ($R$) if for each sequence $\{\mathcal{U}_n:n\in\omega\}$ of open covers of $X$, there is a sequence $\{U_n:n\in\omega\}$ such that for each $n\in\omega$, $U_n\in\mathcal{U}_n$ and $\{U_n:n\in\omega\}$ is an open cover of $X$ (see \cite{R}). A space $X$ is Hurewicz ($H$) if for each sequence $\{\mathcal{U}_n:n\in\omega\}$ of open covers of $X$, there is a sequence $\{\mathcal{V}_n:n\in\omega\}$ such that for each $n\in\omega$, $\mathcal{V}_n$ is a finite subset of $\mathcal{U}_n$ and for each $x\in X$, $x\in\bigcup\mathcal{V}_n$ for all but finitely many $n$ (see \cite{H}). The following star versions for the cases Menger and Rothberger were introduced in \cite{K} and the star versions for the Hurewicz case were defined in \cite{BCK}.

\begin{definition}
A space $X$ is:
\begin{enumerate}
\item star-Menger ($SM$) if for each sequence $\{\mathcal{U}_n:n\in\omega\}$ of open covers of $X$, there is a sequence $\{\mathcal{V}_n:n\in\omega\}$ such that for each $n\in\omega$, $\mathcal{V}_n$ is a finite subset of $\mathcal{U}_n$ and $\{St(\bigcup\mathcal{V}_n,\mathcal{U}_n):n\in\omega\}$ is an open cover of $X$.
\item strongly star-Menger ($SSM$) if for each sequence $\{\mathcal{U}_n:n\in\omega\}$ of open covers of $X$, there exists a sequence $\{F_n:n\in\omega\}$ of finite subsets of $X$ such that $\{St(F_n,\mathcal{U}_n):n\in\omega\}$ is an open cover of $X$.
\item star-Rothberger ($SR$) if for each sequence $\{\mathcal{U}_n:n\in\omega\}$ of open covers of $X$, there are $U_n\in\mathcal{U}_n$, $n\in\omega$, such that $\{St(U_n, \mathcal{U}_n):n\in\omega\}$ is an open cover of $X$.
\item strongly star-Rothberger ($SSR$) if for each sequence $\{\mathcal{U}_n:n\in\omega\}$ of open covers of $X$, there exists a sequence $\{x_n:n\in\omega\}$ of elements of $X$ such that $\{St(x_n, \mathcal{U}_n):n\in\omega\}$ is an open cover of $X$.
\item star-Hurewicz ($SH$) if for each sequence $\{\mathcal{U}_n:n\in\omega\}$ of open covers of $X$, there is a sequence $\{\mathcal{V}_n:n\in\omega\}$ such that for each $n\in\omega$, $\mathcal{V}_n$ is a finite subset of $\mathcal{U}_n$ and for each $x\in X$, $x\in St(\bigcup\mathcal{V}_n,\mathcal{U}_n)$ for all but finitely many $n$.
\item strongly star-Hurewicz ($SSH$) if for each sequence $\{\mathcal{U}_n:n\in\omega\}$ of open covers of $X$, there exists a sequence $\{F_n:n\in\omega\}$ of finite subsets of $X$ such that for each $x\in X$, $x\in St(F_n,\mathcal{U}_n)$ for all but finitely many $n$.
\end{enumerate}
\end{definition}

It is worth to mention that for paracompact Hausdorff spaces the three Menger-type properties, $SM$, $SSM$ and $M$ are equivalent and the same situation holds for the three Rothberger-type properties and the three Hurewicz-type properties (see \cite{K} and \cite{BCK}). Even more, those equivalences still true for paraLindel\"of spaces (see \cite{CGS}).

The following diagram shows the relationships among these properties (in the diagram $C$ and $L$ are used to denote compactness and the Lindel\"{o}f property, respectively).  We mention that none of the arrows in the following diagram reverse. We refer the reader to \cite{K_survey} to see the current state of knowledge about these relationships with others. 

\begin{figure}[h!]
\[
\begin{tikzcd}[row sep=1em, column sep = 1em]
C \arrow[rr] \arrow[dd,swap] && H \arrow[rr] \arrow[dd] && M  \arrow[rr,<-]  \arrow[dr] \arrow[dd] &&
  R \arrow[dd ] \\
&& && & L &&  \\
SSC \arrow[rr] \arrow[dd] && SSH \arrow[rr] \arrow[dd] && SSM \arrow[rr,<-] \arrow[dr]  \arrow[dd]  && SSR \arrow[dd]  \\
&& && & SSL  \arrow[uu,<-,crossing over]&&  \\
SC \arrow[rr] && SH \arrow[rr] && SM \arrow[rr,<-] \arrow[dr] && SR \\
&& && & SL  \arrow[uu,<-,crossing over]&&  
\end{tikzcd}
\]
\caption{Star selection principles.}
\end{figure}
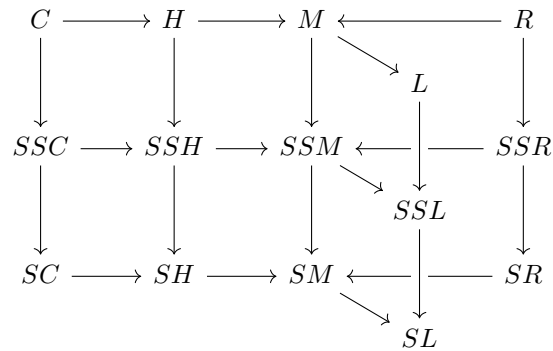

\section{Small unions of some star spaces}\label{sectionsmallunions}

\noindent In \cite{Ta11} Tall proved that if a space $X$ is Lindel\"of and it can be written as a union of less than $\mathfrak{d}$ compact spaces, then $X$ is Menger. It turns out that we can replace ``compact'' by ``star-Hurewicz'' in Tall's result (see Proposition \ref{(L,(<d,SH))} below)\footnote{In \cite{CA}, the authors also use the idea of unions of size less than $\mathfrak{d}$ many Hurewicz-type spaces to obtain some results about star-Scheepers spaces.}, or we can replace ``$\mathfrak{d}$'' and ``compact'' by ``$\mathfrak{b}$'' and  ``star-Menger'' (Proposition \ref{(L,(<b,SM))}). Furthermore, a Lindel\"of space that can be written as a union of less than $\mathfrak{b}$ star-Hurewicz spaces, is Hurewicz (Proposition \ref{(L,(<b,SH))}). These results are contained in Theorem \ref{LindelofUnions} below.\\

\noindent In addition, we investigated what happens if instead of starting with a Lindel\"of space that can be written as some small union, we consider a star-Lindel\"of space or a strongly star-Lindel\"of space or an absolutely strongly star-Lindel\"of space (see Definition \ref{aSSLdefinition} below). Some other interesting relationships were obtained and they are described in Theorem \ref{starLindelofUnions} and Theorem \ref{aSSLindelofUnions} below. Let us first introduce some notation that allows to present these results in an organized manner\footnote{Preliminary versions of some results in this section are contained in the PhD Dissertation of the third listed author (see \cite{GB}).}.

\begin{definition}
Let $X$ be any space, $A$ and $B$ denote some properties and $\kappa$ is some cardinal. $$\big(A,(<\kappa,B)\big)$$
stands for ``$X$ satisfies property $A$ and it can be written as a union of less than $\kappa$ spaces each of them satisfying property $B$''. 
\end{definition}

\noindent For instance, if $L,C$ and $M$ denote Lindel\"of, compact and Menger, respectively, then Tall's result can be written as ``$\big(L,(<\mathfrak{d},C)\big) \rightarrow M$''. More in general, we have:

\begin{theorem}
\label{LindelofUnions}
For any space $X$ the following holds:
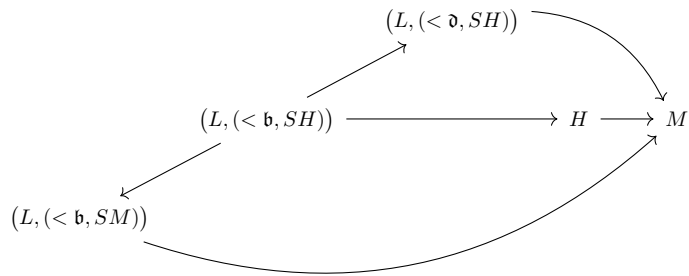
\[
\begin{tikzcd}[row sep=2em, column sep=1em]
&&  \scalebox{0.8}{$\big(L,(<\mathfrak{d},SH)\big)$} \arrow[drrr, bend left ]   \arrow[dl,<-,swap,]&& & \\
&\scalebox{0.8}{$\big(L,(<\mathfrak{b},SH)\big)$}\arrow[rr] \arrow[dl]&&\scalebox{0.8}{$H$} \arrow[rr] &&\scalebox{0.8}{$M$} \\
\scalebox{0.8}{$\big(L,(<\mathfrak{b},SM)\big)$} \arrow[urrrrr, bend right] & && &&
\end{tikzcd}
\]
\end{theorem}

\noindent The proof is divided as Propositions \ref{(L,(<d,SH))}, \ref{(L,(<b,SH))}  and \ref{(L,(<b,SM))}.
\begin{proposition}
\label{(L,(<d,SH))}
If $X$ is a Lindel\"of space and $X$ is the union of less than $\mathfrak{d}$ star-Hurewicz spaces, then $X$ is Menger.
\end{proposition}

\begin{proof}
Let $\kappa$ be a cardinal smaller than $\mathfrak{d}$ and put $X=\bigcup_{\alpha<\kappa}Y_\alpha$ with each $Y_\alpha$ being a star-Hurewicz space. Let $\{ \mathcal{U}_n : n \in \omega \}$ be a sequence of open covers of $X$. Since $X$ is Lindel\"of, we can assume that for each $n \in \omega$, $\mathcal{U}_n$ is countable and put $\mathcal{U}_n = \{U_n^i:i\in \omega\}$. Since each $Y_\alpha$ is star-Hurewicz, for each $\alpha<\kappa$, there exists a finite subset $\mathcal{V}_n^\alpha$ of $\mathcal{U}_n$ such that $\{St(\bigcup\mathcal{V}_n^\alpha, \mathcal{U}_n): n\in\omega\}$ is a $\gamma$-cover of $Y_\alpha$. Define, for each $\alpha<\kappa$, a function $f_\alpha$ as follows: for each $n\in\omega$, let $f_\alpha(n)=min\{i\in\omega: \mathcal{V}_n^\alpha\subseteq\{U_n^j:j\leq i\}\}$. Since the collection $\{f_\alpha: \alpha < \kappa \}$ has size less than $\mathfrak{d}$, there exists $g \in \omega^\omega$ such that for every $\alpha < \kappa$, $g \nleq ^* f_\alpha$. For each $n \in \omega$, let $\mathcal{W}_n =\{U_n^i: i \leq g(n) \}$.\\
\emph{Claim:} $\{St(\bigcup\mathcal{W}_n, \mathcal{U}_n): n\in\omega\}$ is an open cover of $X$.\\
Let $x \in X$. Then, there exists $\alpha < \kappa$ such that $x\in Y_\alpha$. Hence, there is $n_0\in\omega$ so that for every $n\geq n_0$, $x\in St(\bigcup\mathcal{V}_n^\alpha, \mathcal{U}_n)$. Since $g \nleq ^* f_\alpha$, we can take $n> n_0$ such that $g(n)>f_\alpha(n)$. Then $x\in St(\bigcup\mathcal{V}_n^\alpha, \mathcal{U}_n) \subseteq St(\bigcup_{j\leq f_\alpha(k)}U_n^j, \mathcal{U}_n)\subseteq St(\bigcup_{j\leq g(n)}U_n^j, \mathcal{U}_n)= St(\bigcup\mathcal{W}_n, \mathcal{U}_n)$. Therefore, the collection $\{St(\bigcup\mathcal{W}_n, \mathcal{U}_n):n\in\omega\}$ is an open cover of $X$. Thus, $X$ is star-Menger. Finally, since $X$ is Lindel\"{o}f, $X$ is a paracompact space and this allow us to conclude that $X$ is Menger.
\end{proof}

\begin{proposition}
\label{(L,(<b,SH))}
If $X$ is a Lindel\"of space and $X$ is the union of less than $\mathfrak{b}$ star-Hurewicz spaces, then $X$ is Hurewicz.
\end{proposition}

\begin{proof}
Let $\kappa$ be a cardinal smaller than $\mathfrak{b}$ and put $X=\bigcup_{\alpha<\kappa}Y_\alpha$ with each $Y_\alpha$ being a star-Hurewicz space. Let $\{ \mathcal{U}_n : n \in \omega \}$ be a sequence of open covers of $X$. Since $X$ is Lindel\"of, we can assume that for each $n \in \omega$, $\mathcal{U}_n$ is countable and put $\mathcal{U}_n = \{U_n^i:i\in \omega\}$. For each $\alpha<\kappa$, there exists a finite subset $\mathcal{V}_n^\alpha$ of $\mathcal{U}_n$ such that $\{St(\bigcup\mathcal{V}_n^\alpha, \mathcal{U}_n): n\in\omega\}$ is a $\gamma$-cover of $Y_\alpha$. Define, for each $\alpha<\kappa$, a function $f_\alpha$ as follows: for each $n\in\omega$, let $f_\alpha(n)=min\{i\in\omega: \mathcal{V}_n^\alpha\subseteq\{U_n^j:j\leq i\}\}$. Since the collection $\{f_\alpha: \alpha < \kappa \}$ has size less than $\mathfrak{b}$, there exists $g \in \omega^\omega$ such that for every $\alpha < \kappa$, $f_\alpha \leq ^* g$. For each $n \in \omega$, let $\mathcal{W}_n =\{U_n^i: i \leq g(n) \}$.\\
\emph{Claim:} $\{St(\bigcup\mathcal{W}_n, \mathcal{U}_n): n\in\omega\}$ is a $\gamma$-cover of $X$.\\
Let $x \in X$. Then, there exists $\alpha < \kappa$ such that $x\in Y_\alpha$. Hence, there is $n_0\in\omega$ so that for every $n\geq n_0$, $x\in St(\bigcup\mathcal{V}_n^\alpha, \mathcal{U}_n)$. Since $f_\alpha \leq ^* g$, there is $n_1\in\omega$ such that for every $n\geq n_1$, $f_\alpha(n)\leq g(n)$. Put $m= max\{n_0, n_1\}$. Hence, for each $k\geq m$, $x\in St(\bigcup\mathcal{W}_k, \mathcal{U}_k)$. Indeed, let $k\geq m$. Then $x\in St(\bigcup\mathcal{V}_k^\alpha, \mathcal{U}_k) \subseteq St(\bigcup_{j\leq f_\alpha(k)}U_k^j, \mathcal{U}_k)\subseteq St(\bigcup_{j\leq g(k)}U_k^j, \mathcal{U}_k)= St(\bigcup\mathcal{W}_k, \mathcal{U}_k)$. Therefore, the collection $\{St(\bigcup\mathcal{W}_n, \mathcal{U}_n):n\in\omega\}$ is a $\gamma$-cover of $X$. Thus, $X$ is star-Hurewicz. Finally, since $X$ is Lindel\"{o}f, $X$ is a paracompact space and we conclude that $X$ is Hurewicz.
\end{proof}

In this article, by large cover we mean the following: 

\begin{definition}
A cover $\mathcal{U} =\{U_\alpha: \alpha<\kappa\}$ of a space $X$ it's called {\bf large} if for every $\alpha < \kappa$, $\{U_\beta: \alpha\leq\beta<\kappa\}$ is a cover of $X$. We denote the class of large covers of $X$ by $\mathcal{L}(X)$.
\end{definition}

Observe that when we consider countable covers, the previous definition and the one given in \cite{MS1} coincide.\\

For the following lemma we recall some classical notation of (star) selection principles introduced by M. Scheepers (Ko\v{c}inac) in \cite{MS1} (\cite{K}). Let $\mathcal{A}$ and $\mathcal{B}$ be collections of families of sets.\\

$S_{fin}(\mathcal{A},\mathcal{B})$: For each sequence $\{A_n:n\in\omega\}$ of elements of $\mathcal{A}$ there is a sequence $\{B_n:n\in\omega\}$ such that for each $n\in\omega$, $B_n\in [A_n]^{<\omega}$ and $\bigcup\{B_n:n\in\omega\}$ is an element of $\mathcal{B}$.\\

$U_{fin}(\mathcal{A},\mathcal{B})$: For each sequence $\{A_n:n\in\omega\}$ of elements of $\mathcal{A}$ there is a sequence $\{B_n:n\in\omega\}$ such that for each $n\in\omega$, $B_n\in [A_n]^{<\omega}$ and $\{\bigcup B_n:n\in\omega\}$ is an element of $\mathcal{B}$.\\

$S^*_{fin}(\mathcal{A},\mathcal{B})$: For each sequence $\{A_n:n\in\omega\}$ of elements of $\mathcal{A}$ there exists a sequence $\{B_n:n\in\omega\}$ such that for each $n\in\omega$, $B_n \in [A_n]^{< \omega}$ and $\{St(\bigcup B_n,A_n):n\in\omega\}$ is an element of $\mathcal{B}$.\\

The following Lemma will be useful to clarify some steps in the proofs of Propositions \ref{(L,(<b,SM))} and Lemma \ref{AlguitoE**omega}.

\begin{lemma}[Folklore]
\label{largecoverss}
For any space $X$:
\begin{enumerate}
\item\label{largecoverss1} $S_{fin}(\mathcal{O}, \mathcal{O}) \leftrightarrow U_{fin}(\mathcal{O}, \mathcal{L})$.
\item\label{largecoverss2} $S^*_{fin}(\mathcal{O}, \mathcal{O}) \leftrightarrow S^*_{fin}(\mathcal{O}, \mathcal{L})$.
\end{enumerate}
\end{lemma}

\begin{proof}
Let $X$ be any space. Observe that $U_{fin}(\mathcal{O}, \mathcal{L})\rightarrow S_{fin}(\mathcal{O}, \mathcal{O})$  and $S^*_{fin}(\mathcal{O}, \mathcal{L})\rightarrow S^*_{fin}(\mathcal{O}, \mathcal{O})$ are immediate. Now, assume $S_{fin}(\mathcal{O}, \mathcal{O})$ ($S^*_{fin}(\mathcal{O}, \mathcal{O})$ respectively) holds. Let $\{ \mathcal{U}_n : n \in \omega \}$ be any sequence of open covers of $X$ and let $m\in\omega$. Since the collection $\{\mathcal{U}_n : m \leq n < \omega\}$ is a sequence of open covers of $X$, then for each $n\geq m$ there exists a finite subset $\mathcal{V}_ n^m $ of $\mathcal{U}_n$ such that $\bigcup\{\mathcal{V}_n^m : m\leq n < \omega\}$ ($\{St(\bigcup \mathcal{V}_n^m,\mathcal{U}_n) : m\leq n < \omega\}$, resp.) is an open cover of $X$. So, for each $n\in\omega$ we define  $\mathcal{W}_n=\bigcup_{m\leq n}\mathcal{V}_n^m$. Hence, for each $m\in\omega$ the collection $\bigcup\{\mathcal{W}_n:m\leq n <\omega\}$ ($\{St(\bigcup \mathcal{W}_n,\mathcal{U}_n) : m\leq n < \omega\}$, resp.) is an open cover of $X$. Furthermore, for each $m\in\omega$ the collection $\{\bigcup\mathcal{W}_n:m\leq n <\omega\}$ is an open cover of $X$. Thus, for each $n$, $\mathcal{W}_n$ is a finite subset of $\mathcal{U}_n$ and the collection $\{\bigcup\mathcal{W}_n : n < \omega\}$ ($\{St(\bigcup \mathcal{W}_n,\mathcal{U}_n) :  n < \omega\}$, resp.) is a large cover of $X$. Hence, $U_{fin}(\mathcal{O}, \mathcal{L})$ ($S^*_{fin}(\mathcal{O}, \mathcal{L})$, resp.) holds.
\end{proof}

\begin{proposition}
\label{(L,(<b,SM))}
If $X$ is a Lindel\"of space and $X$ is the union of less than $\mathfrak{b}$ star-Menger spaces, then $X$ is Menger.
\end{proposition}
\begin{proof}
Let $\kappa$ be a cardinal smaller than $\mathfrak{b}$ and put $X=\bigcup_{\alpha<\kappa}Y_\alpha$ with each $Y_\alpha$ being a star-Menger space. Let $\{ \mathcal{U}_n : n \in \omega \}$ be a sequence of open covers of $X$. Since $X$ is Lindel\"of, we can assume that for each $n \in \omega$, $\mathcal{U}_n$ is countable and put $\mathcal{U}_n = \{U_n^i:i\in \omega\}$. Since or each $\alpha<\kappa$, $Y_\alpha$ is star-Menger, by Lemma \ref{largecoverss}, for each $\alpha<\kappa$ there is $\mathcal{V}_n^\alpha$ finite subset of $\mathcal{U}_n$ such that for every $m\in\omega$,  $\{St(\bigcup\mathcal{V}_n^\alpha, \mathcal{U}_n): m \leq n < \omega\}$ is an open cover of $Y_\alpha$. Define, for each $\alpha<\kappa$, a function $f_\alpha$ as follows: for each $n\in\omega$, let $f_\alpha(n)=min\{i\in\omega: \mathcal{V}_n^\alpha\subseteq\{U_n^j:j\leq i\}\}$. Since the collection $\{f_\alpha: \alpha < \kappa \}$ has size less than $\mathfrak{b}$, there exists $g \in \omega^\omega$ such that for every $\alpha < \kappa$, $f_\alpha \leq ^* g$. For each $n \in \omega$, let $\mathcal{W}_n =\{U_n^i: i \leq g(n) \}$. \\
\emph{Claim:} $\{St(\bigcup\mathcal{W}_n, \mathcal{U}_n): n\in\omega\}$ is an open cover of $X$.\\
Let $x \in X$ and fix $\alpha < \kappa$ such that $x\in Y_\alpha$. Hence, there is $m\in\omega$ so that for every $n\geq m$, $f_\alpha(n)\leq g(n)$. Let $n \geq m$ such that $x\in St(\bigcup\mathcal{V}_n^\alpha, \mathcal{U}_n)$. Observe that $St(\bigcup\mathcal{V}_n^\alpha, \mathcal{U}_n) \subseteq St(\bigcup_{j\leq f_\alpha(n)}U_n^j, \mathcal{U}_n)\subseteq St(\bigcup_{j\leq g(n)}U_n^j, \mathcal{U}_n)= St(\bigcup\mathcal{W}_n, \mathcal{U}_n)$. Therefore, the collection $\{St(\bigcup\mathcal{W}_n, \mathcal{U}_n):n\in\omega\}$ is an open cover of $X$. Thus, $X$ is star-Menger. Finally, since $X$ is Lindel\"{o}f, $X$ is a paracompact space and this allow us to conclude that $X$ is Menger.
\end{proof}

Following \cite{BH} (see also \cite{RZ}), we recall some modifications of the Menger and Hurewicz properties, called $E^*_\omega$ and $E^{**}_\omega$ properties, respectively. We say that a space $X$ has the property $E^*_\omega$ ($E^{**}_\omega$), if for every sequence $\{\mathcal{U}_n:n\in\omega\}$ of countable open covers of $X$, there exists a sequence $\{\mathcal{V}_n:n\in\omega\}$ such that for each $n\in\omega$, $\mathcal{V}_n$ is a finite subset of $\mathcal{U}_n$ and $\{\bigcup\mathcal{V}_n:n\in\omega\}$ is an open cover ($\gamma$-cover) of $X$. If we recall the definition of countably compact space (every countable cover has a finite subcover), properties $E^*_\omega$ and $E^{**}_\omega$ could be called countably Menger and countably Hurewicz, respectively.  As pointed out in \cite{RZ}, in the class of Lindel\"{o}f spaces, the properties Menger and $E^*_\omega$ are the same. However, this fact is not true in general. The space $\omega_1$ (with the order topology) has the property $E^*_\omega$ and it is not a Menger space.

\begin{lemma}
\label{AlgoE*omega}
If $X$ can be written as a union of less than $\mathfrak{d}$ many Hurewicz spaces, then $X$ is $E^*_\omega$.
\end{lemma}

\begin{proof}
Assume $X = \bigcup_{\alpha < \kappa}H_\alpha$ so that $\kappa < \mathfrak{d}$ and each $H_\alpha$ is Hurewicz. Let $\{\mathcal{U}_n:n \in \omega\}$ be a sequence of countable open covers of $X$. For each $n \in \omega$ let $\mathcal{U}_n = \{U^i_n:i \in \omega\}$. For each $\alpha < \kappa$ and $ n \in \omega$, there exists $\mathcal{F}^\alpha_n \in [\mathcal{U}_n]^{<\omega}$ so that $\{\bigcup \mathcal{F}^\alpha_n : n \in \omega\}$ is a $\gamma$-cover of $H_\alpha$. For each $\alpha < \kappa$ and $ n \in \omega$, let $f_\alpha(n) = min\{m\in \omega: \mathcal{F}^\alpha_n \subseteq \{U^i_n: i \leq m\}\}$. Given that $\{f_\alpha : \alpha < \kappa\}$ has size less than $\mathfrak{d}$, there exists $g \in \omega ^ \omega$ so that for each $\alpha < \kappa$, $g \not \leq ^* f_\alpha$. For each $n \in \omega$, let $G_n = \{U^i_n: i \leq g(n)\}$. Let us check that $\{\bigcup G_n:n \in \omega\}$ is an open cover of $X$. Let $x \in X$, then there is $\alpha < \kappa$ so that $x \in H_\alpha$. Thus, there is $n_0 \in \omega$ so that for each $n \geq n_0$, $x \in \bigcup \mathcal{F}^\alpha_n$. Pick $m \geq n_0$ with $f_\alpha(m) < g(m)$. Hence, $x \in \bigcup \mathcal{F}^\alpha_m \subseteq \bigcup G_m$. Therefore, $X$ is $E^*_\omega$.
\end{proof}

Analogous to Lemma \ref{AlgoE*omega}, the following also holds:

\begin{lemma}
\label{AlguitoE**omega}
\begin{enumerate}
\item \label{AlguitoE**omega_b_Menger} If $X$ can be written as a union of less than $\mathfrak{b}$ many Menger spaces, then $X$ is $E^{*}_\omega$.
\item \label{AlguitoE**omega_b_Hurewicz} If $X$ can be written as a union of less than $\mathfrak{b}$ many Hurewicz spaces, then $X$ is $E^{**}_\omega$.
\end{enumerate}
\end{lemma}

\begin{proof}
$(1)$ Assume $X = \bigcup_{\alpha < \kappa}M_\alpha$ so that $\kappa < \mathfrak{b}$ and each $M_\alpha$ is Menger. Let $\{\mathcal{U}_n:n \in \omega\}$ be a sequence of countable open covers of $X$. For each $n \in \omega$ let $\mathcal{U}_n = \{U^i_n:i \in \omega\}$. By Lemma \ref{largecoverss}, for each $\alpha < \kappa$ and $ n \in \omega$, there exists $\mathcal{F}^\alpha_n \in [\mathcal{U}_n]^{<\omega}$ so that $\{\bigcup \mathcal{F}^\alpha_n : n \in \omega\}$ is a large cover of $M_\alpha$. For each $\alpha < \kappa$ and $ n \in \omega$, let $f_\alpha(n) = min\{m\in \omega: \mathcal{F}^\alpha_n \subseteq \{U^i_n: i \leq m\}\}$. Given that $\{f_\alpha : \alpha < \kappa\}$ has size less than $\mathfrak{b}$, there exists $g \in \omega ^ \omega$ so that for each $\alpha < \kappa$, $f_\alpha\leq ^* g$. For each $n \in \omega$, let $G_n = \{U^i_n: i \leq g(n)\}$. Let us check that $\{\bigcup G_n:n \in \omega\}$ is an open cover of $X$. Let $x \in X$, then there is $\alpha < \kappa$ so that $x \in M_\alpha$. Thus, there is $n_0 \in \omega$ so that for each $n \geq n_0$, $f_\alpha(n)\leq g(n)$. Since $\{\bigcup \mathcal{F}^\alpha_n : n \in \omega\}$ is a large cover of $M_\alpha$, there is $n\geq n_0$ such that $x \in \bigcup \mathcal{F}^\alpha_n\subseteq \bigcup_{i\leq f_\alpha(n)} U^i_n\subseteq \bigcup_{i\leq g(n)} U^i_n=\bigcup G_n$. Therefore, $X$ is $E^*_\omega$.

 $(2)$ Assume $X = \bigcup_{\alpha < \kappa}H_\alpha$ so that $\kappa < \mathfrak{b}$ and each $H_\alpha$ is Hurewicz. Let $\{\mathcal{U}_n:n \in \omega\}$ be a sequence of countable open covers of $X$. For each $n \in \omega$ let $\mathcal{U}_n = \{U^i_n:i \in \omega\}$. For each $\alpha < \kappa$ and $ n \in \omega$, there exists $\mathcal{F}^\alpha_n \in [\mathcal{U}_n]^{<\omega}$ so that $\{\bigcup \mathcal{F}^\alpha_n : n \in \omega\}$ is $\gamma$-cover of $H_\alpha$. For each $\alpha < \kappa$ and $ n \in \omega$, let $f_\alpha(n) = min\{m\in \omega: \mathcal{F}^\alpha_n \subseteq \{U^i_n: i \leq m\}\}$. Given that $\{f_\alpha : \alpha < \kappa\}$ has size less than $\mathfrak{b}$, there exists $g \in \omega ^ \omega$ so that for each $\alpha < \kappa$, $f_\alpha\leq ^* g$. For each $n \in \omega$, let $G_n = \{U^i_n: i \leq g(n)\}$. Let us check that $\{\bigcup G_n:n \in \omega\}$ is a $\gamma$-cover of $X$. Let $x \in X$, then there is $\alpha < \kappa$ so that $x \in H_\alpha$. Thus, there is $n_0 \in \omega$ so that for each $n \geq n_0$, $f_\alpha(n)\leq g(n)$. Since $\{\bigcup \mathcal{F}^\alpha_n : n \in \omega\}$ is a $\gamma$-cover of $H_\alpha$, we can fix $n_1\geq n_0$ such that for each $n \geq n_1$, $x \in \bigcup \mathcal{F}^\alpha_n\subseteq \bigcup_{i\leq f_\alpha(n)} U^i_n\subseteq \bigcup_{i\leq g(n)} U^i_n=\bigcup G_n$. Therefore, $X$ is $E^{**}_\omega$.
\end{proof}

\begin{theorem}
\label{starLindelofUnions}
For any space $X$ the following holds:
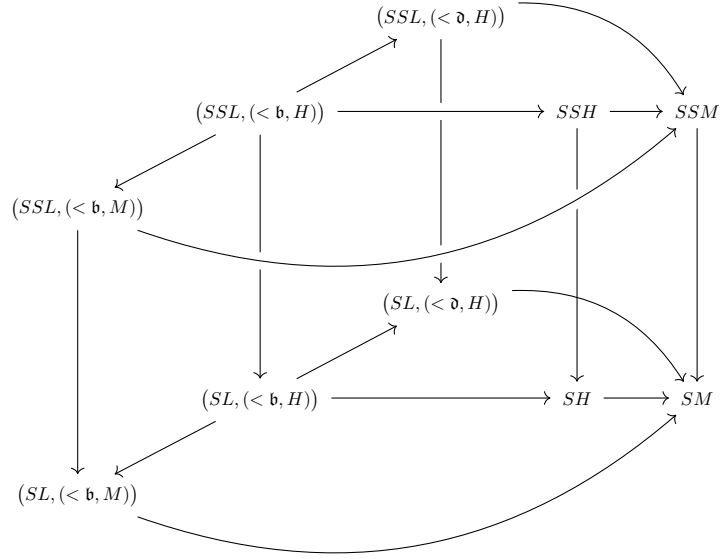
\[
\begin{tikzcd}[row sep=2em, column sep=1em]
&&  \scalebox{0.7}{$\big(SSL,(<\mathfrak{d},H)\big)$} \arrow[drrr, bend left ]  \arrow[ddd] \arrow[dl,<-]&& & \\
&\scalebox{0.7}{$\big(SSL,(<\mathfrak{b},H)\big)$}\arrow[rr,crossing over] \arrow[ddd] \arrow[dl]&&\scalebox{0.7}{$SSH$} \arrow[rr] \arrow[ddd,crossing over] &&\scalebox{0.7}{$SSM$} \arrow[ddd]\\
\scalebox{0.7}{$\big(SSL,(<\mathfrak{b},M)\big)$} \arrow[urrrrr, crossing over, bend right] \arrow[ddd]& && &&\\
&& \scalebox{0.7}{$ \big(SL,(<\mathfrak{d},H)\big)$} \arrow[drrr, bend left ]  \arrow[dl,<-] & &&\\
& \scalebox{0.7}{$\big(SL,(<\mathfrak{b},H)\big)$} \arrow[rr]  \arrow[dl]&&\scalebox{0.7}{$SH$} \arrow[rr]&&\scalebox{0.7}{$SM$}\\
 \scalebox{0.7}{$\big(SL,(<\mathfrak{b},M)\big)$} \arrow[urrrrr, bend right]& && &&
\end{tikzcd}
\]
\end{theorem}

\noindent The proof is divided as Propositions \ref{(SSL,(<d,H))}, \ref{(SSL,(<b,H))}, \ref{(SSL,(<b,M))}, \ref{(SL,(<d,H))}, \ref{(SL,(<b,H))} and \ref{(SL,(<b,M))}.

\begin{proposition}
\label{(SSL,(<d,H))}
If $X$ is a strongly star-Lindel\"of space and $X$ is the union of less than $\mathfrak{d}$ Hurewicz spaces, then $X$ is strongly star-Menger.
\end{proposition}

\begin{proof}
Let $\kappa$ be any cardinal smaller than $\mathfrak{d}$ and put $X=\bigcup_{\alpha<\kappa}H_\alpha$ with each $H_\alpha$ being a Hurewicz space. Let $\{ \mathcal{U}_n : n \in \omega \}$ be a sequence of open covers of $X$. Since $X$ is strongly star-Lindel\"of, for each $n \in \omega$ there exists $C_n\in [X]^{\leq\omega}$ such that $St(C_n,\mathcal{U}_n) = X$. For each $n\in\omega$, put $C_n = \{x_n^i:i\in \omega\}$. Note that $St(C_n,\mathcal{U}_n)=\bigcup_{i\in\omega}St(x_n^i,\mathcal{U}_n)$  for each $n\in\omega$. So, for each $n\in\omega$, the collection $\mathcal{W}_n = \{St(x_n^i,\mathcal{U}_n):i\in\omega\}$ is a countable open cover of $X$. By Lemma \ref{AlgoE*omega}, $X$ is $E^*_\omega$ and then we can get finite subcollections $\mathcal{F}_n$ of $\mathcal{W}_n$ so that $\{\bigcup\mathcal{F}_n:n\in\omega\}$ is a cover of $X$. Equivalently, we get finite subsets $F_n$ of $X$ such that the collection $\{St(F_n, \mathcal{U}_n):n\in\omega\}$ is an open cover of $X$. Thus, $X$ is strongly star-Menger.
\end{proof}

\begin{proposition}
\label{(SSL,(<b,H))}
If $X$ is a strongly star-Lindel\"of space and $X$ is the union of less than $\mathfrak{b}$ Hurewicz spaces, then $X$ is strongly star-Hurewicz.
\end{proposition}

\begin{proof}
Let $\kappa$ be any cardinal smaller than $\mathfrak{b}$ and put $X=\bigcup_{\alpha<\kappa}H_\alpha$ with each $H_\alpha$ being a Hurewicz space. Let $\{ \mathcal{U}_n : n \in \omega \}$ be a sequence of open covers of $X$. Since $X$ is strongly star-Lindel\"of, for each $n \in \omega$ there exists $C_n\in [X]^{\leq\omega}$ such that $St(C_n,\mathcal{U}_n) = X$. For each $n\in\omega$, put $C_n = \{x_n^i:i\in \omega\}$. Note that $St(C_n,\mathcal{U}_n)=\bigcup_{i\in\omega}St(x_n^i,\mathcal{U}_n)$  for each $n\in\omega$. So, for each $n\in\omega$, the collection $\mathcal{W}_n = \{St(x_n^i,\mathcal{U}_n):i\in\omega\}$ is a countable open cover of $X$. By Lemma \ref{AlguitoE**omega} (\ref{AlguitoE**omega_b_Hurewicz}), $X$ is $E^{**}_\omega$ and then we can get finite subcollections $\mathcal{F}_n$ of $\mathcal{W}_n$ so that $\{\bigcup\mathcal{F}_n:n\in\omega\}$ is a $\gamma$-cover of $X$. Equivalently, we get finite subsets $F_n$ of $X$ such that the collection $\{St(F_n, \mathcal{U}_n):n\in\omega\}$ is a $\gamma$-cover of $X$. Thus, $X$ is strongly star-Hurewicz.
\end{proof}

\begin{proposition}
\label{(SSL,(<b,M))}
If $X$ is a strongly star-Lindel\"of space and $X$ is the union of less than $\mathfrak{b}$ Menger spaces, then $X$ is strongly star-Menger.
\end{proposition}

\begin{proof}

Let $\kappa$ be any cardinal smaller than $\mathfrak{b}$ and put $X=\bigcup_{\alpha<\kappa}M_\alpha$ with each $M_\alpha$ being a Menger space. Let $\{\mathcal{U}_n : n \in \omega \}$ be a sequence of open covers of $X$. Since $X$ is strongly star-Lindel\"of, for each $n \in \omega$ there exists $C_n\in [X]^{\leq\omega}$ such that $St(C_n,\mathcal{U}_n) = X$. For each $n\in\omega$, put $C_n = \{x_n^i:i\in \omega\}$. Observe that for each $n\in\omega$, $St(C_n,\mathcal{U}_n)=\bigcup_{i\in\omega}St(x_n^i,\mathcal{U}_n)$. Thus, for each $n\in\omega$, the collection $\mathcal{W}_n=\{St(x_n^i,\mathcal{U}_n):i\in\omega\}$ is a countable open cover of $X$. By Lemma \ref{AlguitoE**omega} (\ref{AlguitoE**omega_b_Menger}), $X$ is $E^*_\omega$ and then,  we can get finite collections $\mathcal{F}_n$ of $\mathcal{W}_n$ such that the collection $\{\bigcup\mathcal{F}_n:n\in\omega\}$ is an open cover of $X$. Equivalently, we get finite subsets $F_n$ of $X$ such that $\{St(F_n,\mathcal{U}_n):n\in\omega\}$ is an open cover of $X$. Thus, $X$ is strongly star-Menger.
\end{proof}

\begin{proposition}
\label{(SL,(<d,H))}
If $X$ is a star-Lindel\"of space and $X$ is the union of less than $\mathfrak{d}$ Hurewicz spaces, then $X$ is star-Menger.
\end{proposition}

\begin{proof}

Let $\kappa$ be any cardinal smaller than $\mathfrak{d}$ and put $X=\bigcup_{\alpha<\kappa}H_\alpha$ with each $H_\alpha$ being a Hurewicz space. Let $\{ \mathcal{U}_n : n \in \omega \}$ be a sequence of open covers of $X$. Since $X$ is star-Lindel\"of, for each $n \in \omega$ there exists $\mathcal{V}_n\in [\mathcal{U}_n]^{\leq\omega}$ such that $St(\bigcup\mathcal{V}_n,\mathcal{U}_n) = X$. For each $n\in\omega$, put $\mathcal{V}_n = \{V_n^i:i\in \omega\}$. Note that $St(\bigcup\mathcal{V}_n,\mathcal{U}_n)=\bigcup_{i\in\omega}St(V_n^i,\mathcal{U}_n)$  for each $n\in\omega$. So, for each $n\in\omega$, the collection $\mathcal{W}_n = \{St(V_n^i,\mathcal{U}_n):i\in\omega\}$ is a countable open cover of $X$. By Lemma \ref{AlgoE*omega}, $X$ is $E^*_\omega$ and then we can get finite subcollections $\mathcal{H}_n$ of $\mathcal{W}_n$ so that $\{\bigcup\mathcal{H}_n:n\in\omega\}$ is an open cover of $X$. That is, we get finite sets $\mathcal{F}_n$ of $\mathcal{U}_n$ such that the collection $\{St(\bigcup\mathcal{F}_n, \mathcal{U}_n):n\in\omega\}$ is an open cover of $X$. Thus, $X$ is star-Menger.
\end{proof}

\begin{proposition}
\label{(SL,(<b,H))}
If $X$ is a star-Lindel\"of space and $X$ is the union of less than $\mathfrak{b}$ Hurewicz spaces, then $X$ is star-Hurewicz.
\end{proposition}

\begin{proof}
Let $\kappa$ be any cardinal smaller than $\mathfrak{b}$ and put $X=\bigcup_{\alpha<\kappa}H_\alpha$ with each $H_\alpha$ being a Hurewicz space. Let $\{ \mathcal{U}_n : n \in \omega \}$ be a sequence of open covers of $X$. Since $X$ is star-Lindel\"of, for each $n \in \omega$ there exists $\mathcal{V}_n\in [\mathcal{U}_n]^{\leq\omega}$ such that $St(\bigcup\mathcal{V}_n,\mathcal{U}_n) = X$. For each $n\in\omega$, put $\mathcal{V}_n = \{V_n^i:i\in \omega\}$. Note that $St(\bigcup\mathcal{V}_n,\mathcal{U}_n)=\bigcup_{i\in\omega}St(V_n^i,\mathcal{U}_n)$  for each $n\in\omega$. So, for each $n\in\omega$, the collection $\mathcal{W}_n = \{St(V_n^i,\mathcal{U}_n):i\in\omega\}$ is a countable open cover of $X$. By Lemma \ref{AlguitoE**omega} (\ref{AlguitoE**omega_b_Hurewicz}), $X$ is $E^{**}_\omega$ and then we can get finite subcollections $\mathcal{H}_n$ of $\mathcal{W}_n$ so that $\{\bigcup\mathcal{H}_n:n\in\omega\}$ is a $\gamma$-cover of $X$. That is, we get finite sets $\mathcal{F}_n$ of $\mathcal{U}_n$ such that the collection $\{St(\bigcup\mathcal{F}_n, \mathcal{U}_n):n\in\omega\}$ is a $\gamma$-cover of $X$. Thus, $X$ is star-Hurewicz.
\end{proof}

\begin{proposition}
\label{(SL,(<b,M))}
If $X$ is a star-Lindel\"of space and $X$ is the union of less than $\mathfrak{b}$ Menger spaces, then $X$ is star-Menger.
\end{proposition}

\begin{proof}
Let $\kappa$ be any cardinal smaller than $\mathfrak{b}$ and put $X=\bigcup_{\alpha<\kappa}M_\alpha$ with each $M_\alpha$ being a Menger space. Let $\{ \mathcal{U}_n : n \in \omega \}$ be a sequence of open covers of $X$. Since $X$ is star-Lindel\"of, for each $n \in \omega$ there exists $\mathcal{V}_n\in [\mathcal{U}_n]^{\leq\omega}$ such that $St(\bigcup\mathcal{V}_n,\mathcal{U}_n) = X$. For each $n\in\omega$, put $\mathcal{V}_n = \{V_n^i:i\in \omega\}$. Note that $St(\bigcup\mathcal{V}_n,\mathcal{U}_n)=\bigcup_{i\in\omega}St(V_n^i,\mathcal{U}_n)$  for each $n\in\omega$. So, for each $n\in\omega$, the collection $\mathcal{W}_n = \{St(V_n^i,\mathcal{U}_n):i\in\omega\}$ is an countable open cover of $X$. By Lemma \ref{AlguitoE**omega} (\ref{AlguitoE**omega_b_Menger}), $X$ is $E^*_\omega$ and then we can get finite subcollections $\mathcal{H}_n$ of $\mathcal{W}_n$ so that $\{\bigcup\mathcal{H}_n:n\in\omega\}$ is an open cover of $X$. Equivalently, we obtain finite subcollections $\mathcal{F}_n$ of $\mathcal{U}_n$ so that the collection $\{St(\bigcup\mathcal{F}_n,\mathcal{U}_n):n\in\omega\}$ is an open cover of $X$. Thus, $X$ is star-Menger.
\end{proof}

Another interesting fact is that we also have theorems of same structure for the selective versions of the star selection principles. We recall the necessary definitions to state the analogous theorem for these selective versions.

The following version of the strongly star-Lindel\"{o}f property was introduced and studied by Bonanzinga in \cite{Bo}: 

\begin{definition}\label{aSSLdefinition}
A space $X$ is absolutely strongly star-Lindel\"{o}f ($aSSL$) if for any open cover $\mathcal{U}$ of $X$ and any dense subset $D$ of $X$, there is a countable set $C\subseteq D$ such that $St(C,\mathcal{U})=X$.
\end{definition}

The selective versions below are stronger properties than the classical star selection principles (see for instance \cite{CG} for more information of these properties\footnote{The Rothberger case and some other interesting properties are also given in \cite{CG}}).

\begin{definition}
We say that a space $X$ is:
\begin{enumerate}
    \item selectively strongly star-Menger ($selSSM$) if for each sequence $\{\mathcal{U}_n:n\in\omega\}$ of open covers of $X$ and each sequence $\{D_n:n\in\omega\}$ of dense sets of $X$, there exists a sequence $\{F_n:n\in\omega\}$ of finite sets such that $F_n\subseteq D_n$, $n\in\omega$, and $\{St(F_n,\mathcal{U}_n):n\in\omega\}$ is an open cover of $X$ (see \cite{Cuz}).
    \item selectively strongly star-Hurewicz ($selSSH$) if for each sequence $\{\mathcal{U}_n:n\in\omega\}$ of open covers of $X$ and each sequence $\{D_n:n\in\omega\}$ of dense sets of $X$, there exists a sequence $\{F_n:n\in\omega\}$ of finite sets such that $F_n\subseteq D_n$, $n\in\omega$, and $\{St(F_n,\mathcal{U}_n):n\in\omega\}$ is a $\gamma$-cover of $X$.
\end{enumerate}
\end{definition}

\begin{theorem}
\label{aSSLindelofUnions}
For any space $X$ the following holds:
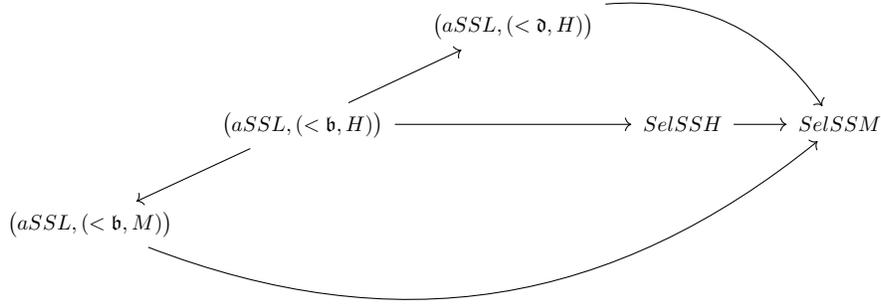
\[
\begin{tikzcd}[row sep=2em, column sep=1em]
&&  \scalebox{0.8}{$\big(aSSL,(<\mathfrak{d},H)\big)$} \arrow[drrr, bend left ]   \arrow[dl,<-,swap,]&& & \\
&\scalebox{0.8}{$\big(aSSL,(<\mathfrak{b},H)\big)$}\arrow[rr] \arrow[dl]&&\scalebox{0.8}{$SelSSH$} \arrow[rr] &&\scalebox{0.8}{$SelSSM$} \\
\scalebox{0.8}{$\big(aSSL,(<\mathfrak{b},M)\big)$} \arrow[urrrrr, bend right] & && &&
\end{tikzcd}
\]
\end{theorem}

\noindent The proof is divided as Propositions \ref{(aSSL,(<d,H))}, \ref{(aSSL,(<b,H))}, \ref{(aSSL,(<b,M))}.

\begin{proposition}
\label{(aSSL,(<d,H))}
If $X$ is an absolutely strongly star-Lindel\"of space and $X$ is the union of less than $\mathfrak{d}$ Hurewicz spaces, then $X$ is selectively strongly star-Menger.
\end{proposition}

\begin{proof}
Assume $X$ is $aSSL$ and let $\kappa < \mathfrak{d}$ such that $X = \bigcup _{\alpha < \kappa}H_\alpha$ with each $H_\alpha$ being a Hurewicz space.
Let $\{\mathcal{U}_n: n\in \omega\}$ be any sequence of open covers and let $\{D_n:n\in \omega\}$ be any sequence of dense subsets of $X$. For each $n \in \omega$ fix  $E_n =\{d_q^n:q\in \omega\}\in [D_n]^{\leq \omega}$ such that $St(E_n,\mathcal{U}_n) = X$.
Thus, for each $n\in \omega$, the collection $\{St(d_q^n,\mathcal{U}_n):q\in \omega\}$ is a countable open cover of $X$. In a similar fashion as Proposition \ref{(SSL,(<d,H))}, we use Lemma \ref{AlgoE*omega} to conclude the proof.
\end{proof}

\begin{proposition}
\label{(aSSL,(<b,H))}
If $X$ is an absolutely strongly star-Lindel\"of space and $X$ is the union of less than $\mathfrak{b}$ Hurewicz spaces, then $X$ is selectively strongly star-Hurewicz.
\end{proposition}

\begin{proof}
Assume $X$ is $aSSL$ and let $\kappa < \mathfrak{b}$ such that $X = \bigcup _{\alpha < \kappa}H_\alpha$ with each $H_\alpha$ being a Hurewicz space.
Let $\{\mathcal{U}_n: n\in \omega\}$ be any sequence of open covers and let $\{D_n:n\in \omega\}$ be any sequence of dense subsets of $X$. For each $n \in \omega$ fix  $E_n =\{d_q^n:q\in \omega\}\in [D_n]^{\leq \omega}$ such that $St(E_n,\mathcal{U}_n) = X$.
Thus, for each $n\in \omega$, the collection $\{St(d_q^n,\mathcal{U}_n):q\in \omega\}$ is a countable open cover of $X$. In a similar fashion as Proposition \ref{(SSL,(<b,H))}, we use Lemma \ref{AlguitoE**omega} (\ref{AlguitoE**omega_b_Hurewicz}) to complete the proof.
\end{proof}

\begin{proposition}
\label{(aSSL,(<b,M))}
If $X$ is an absolutely strongly star-Lindel\"of space and $X$ is the union of less than $\mathfrak{b}$ Menger spaces, then $X$ is selectively strongly star-Menger.
\end{proposition}

\begin{proof}
Assume $X$ is $aSSL$ and let $\kappa < \mathfrak{b}$ such that $X = \bigcup _{\alpha < \kappa}M_\alpha$ with each $M_\alpha$ being a Menger space.
Let $\{\mathcal{U}_n: n\in \omega\}$ be any sequence of open covers and let $\{D_n:n\in \omega\}$ be any sequence of dense subsets of $X$. For each $n \in \omega$ fix  $E_n =\{d_q^n:q\in \omega\}\in [D_n]^{\leq \omega}$ such that $St(E_n,\mathcal{U}_n) = X$.
Thus, for each $n\in \omega$, the collection $\{St(d_q^n,\mathcal{U}_n):q\in \omega\}$ is a countable open cover of $X$. In a similar way as Proposition \ref{(SSL,(<b,M))}, we use Lemma \ref{AlguitoE**omega} (\ref{AlguitoE**omega_b_Menger}) to conclude the proof.
\end{proof}

\section{Iterated stars}\label{sectioniteratedstars}

A well-known study about iterations of star versions of Lindel\"of properties was made in \cite{DRRT} (see also \cite{M}). 
The previous section motivated a similar study for star selection properties. We start giving some results that involve refinements of open covers.

In \cite{DRRT} the properties $n$-star Lindel\"{o}f  and strongly $n$-star Lindel\"{o}f are defined and the authors show that every $n$-star Lindel\"{o}f space is strongly $n+1$-star-Lindel\"{o}f (Theorem 3.1.1 (3)). In the class of metaLindel\"{o}f spaces the converse holds:

\begin{proposition}\label{ML+(n+1)SSLimplies(n)SL}
If $X$ is metaLindel\"{o}f and strongly $n+1$-star Lindel\"{o}f then $X$ is $n$-star Lindel\"{o}f.
\end{proposition}

\begin{proof}
 Let $\mathcal{U}$ be an open cover of $X$. Since $X$ is metaLindel\"{o}f we can assume that $\mathcal{U}$ is point-countable. Let $C \in [X]^\omega$ so that $St^n\big(St(C,\mathcal{U}),\mathcal{U}\big) = X$. Since $\mathcal{U}$ is point-countable, there exists $W \in [\mathcal{U}]^{\omega}$ so that $St(C,\mathcal{U}) \subseteq \bigcup W$. Hence, $St^{n+1}(C,\mathcal{U}) \subseteq St^n(\bigcup W, \mathcal{U})$ i.e., $X$ is $n$-star-Lindel\"{o}f.
\end{proof}

\begin{lemma}[Folklore]
\label{containmentsOfStars}
Let $\mathcal{U}$ be an open cover of a topological space $X$. If $A \subseteq B \subseteq X$ and $\mathcal{V}$ is a refinement of $\mathcal{U}$, then for each $n \in \omega$, $St^n(A,\mathcal{V}) \subseteq St^n(B , \mathcal{U})$.
\end{lemma}

\begin{proposition}
\label{ParalindelofnStarLindelofIsnstronglystarLindelof}
Every Paralindel\"of and $n$-star Lindel\"of  space $X$ is $n$-strongly star Lindel\"of for each $n \in \omega$.
\end{proposition}

\begin{proof}
Assume $X$ is a Paralindel\"of and $n$-star Lindel\"of space for some $n \in \omega$. Let $\mathcal{U}$ be an open cover of $X$. Without loss of generality we can assume that  $\mathcal{U}$ is locally countable. For each $x \in X$, let $U_x$ be an open set so that $|\{U \in \mathcal{U}: U_x \cap U \neq \emptyset\}| \leq \omega$ and $\mathcal{V} := \{U_x :x \in X\}$ refines $\mathcal{U}$. Since $X$ is $n$-star Lindel\"of, there exists $C \in [X]^{\leq \omega}$ so that $St^n(\bigcup _{x\in C}U_x,\mathcal{V}) = X$. \\
For $x \in C$, let $\mathcal{W}_x = \{U \in \mathcal{U}:U_x \cap U \neq \emptyset\}$. Then $|\mathcal{W}_x|\leq \omega$. If $\mathcal{W} = \bigcup _{x \in C}\mathcal{W}_x$, then $|\mathcal{W}| \leq \omega$. Thus, for each $W \in \mathcal{W}$, fix $y_W \in W$.\\
Claim: $St(\bigcup _{x\in C}U_x,\mathcal{V}) \subseteq St(\{y_W:W \in \mathcal{W}\},\mathcal{U})$.\\ Indeed, let $y \in St(\bigcup _{x\in C}U_x,\mathcal{V})$, hence, there exists $V\in \mathcal{V}$ so that $y \in V$ and $V\cap \bigcup _{x\in C}U_x \neq \emptyset$. Let $U \in \mathcal{U}$ so that $V \subseteq U$. Then, $U \in \mathcal{W}$. Thus, $y \in St(y_U,\mathcal{U}) \subseteq St(\{y_W:W \in \mathcal{W}\},\mathcal{U})$.\\
Since $\mathcal{V}$ is a refinement o  $\mathcal{U}$, and $St(\bigcup _{x\in C}U_x,\mathcal{V}) \subseteq St(\{y_W:W \in \mathcal{W}\},\mathcal{U})$, by Lemma \ref{containmentsOfStars}, $X = St^n(\bigcup _{x\in C}U_x,\mathcal{V}) \subseteq St^n(\{y_W:W \in \mathcal{W}\},\mathcal{U})$. That is, $X$ is $n$-strongly star  Lindel\"of.
\end{proof}

Observe that if we have a paraLindel\"of space $X$ that is either $n$-star Lindel\"of or $n$-strongly star Lindel\"of for some $n \in \omega$, we can use Proposition \ref{ParalindelofnStarLindelofIsnstronglystarLindelof} and Proposition \ref{ML+(n+1)SSLimplies(n)SL}, as many times as needed to get that $X$ is star-Lindel\"of. Since paraLindel\"of, star-Lindel\"of spaces are Lindel\"of (check Theorem 2.6 in \cite{CGS} or Theorem 2.24 in \cite{YKS2}), we have the following:

\begin{corollary}
\label{paraLIndelofColapsesnStarLindelof}
In the class of paraLindel\"of spaces, for each $n \in \omega$, the following properties are equivalent:
\begin{enumerate}
    \item[(i)] Lindel\"of;
    \item[(ii)] strongly $n$-star Lindel\"of;
    \item[(iii)] $n$-star Lindel\"of.
\end{enumerate}
\end{corollary}

In Proposition 53 of \cite{M}, Matveev shows that every metaLindel\"{o}f strongly $2$-star Lindel\"{o}f is absolutely strongly $2$-star Lindel\"{o}f (he calls star Lindel\"{o}f what we call strongly star Lindel\"{o}f). Thus, we can ask the general case:

\begin{question}
Is it true that every metaLindel\"of strongly $n$-star Lindel\"of space is absolutely strongly $n$-star Lindel\"of?
\end{question}

Additionally, since both properties, absolutely strongly $n+1$-star Lindel\"of and $n$-star Lindel\"{o}f are stronger than strongly $n+1$-star Lindel\"of, it is worth to investigate the following:

\begin{question}
What is the relationship between the properties absolutely strongly $n+1$-star Lindel\"of and $n$-star Lindel\"{o}f?
\end{question}

We introduce similar definitions for the star versions of the Menger property:\\
\begin{definition}\label{nstarMenger}
\begin{enumerate}
    \item A space $X$ is called $n$-star-Menger if for every sequence of open covers  $\{ \mathcal{U}_n : n \in \omega \}$ there exists a sequence  $\{ \mathcal{V}_n : n \in \omega \}$ so that for each $n \in \omega$ $\mathcal{V}_n \in [\mathcal{U}_n]^{< \omega}$ and $\{St^n(\bigcup \mathcal{V}_n,\mathcal{U}_n):n \in \omega\}$ is an open cover of $X$.
    \item  A space $X$ is called strongly $n$-star-Menger if for every sequence of open covers  $\{ \mathcal{U}_n : n \in \omega \}$ there exists a sequence  $\{ F_n : n \in \omega \}$ so that for each $n \in \omega$ $F_n \in [X]^{< \omega}$ and $\{St^n(F_n,\mathcal{U}_n):n \in \omega\}$ is an open cover of $X$.
\end{enumerate}

Observe that $1$-star-Menger and strongly $1$-star-Menger spaces are precisely star-Menger and strongly star-Menger spaces, respectively.
\end{definition}

It follows immediately from Proposition \ref{ML+(n+1)SSLimplies(n)SL} that if $X$ is metaLindel\"{o}f and strongly $n+1$-star Menger then $X$ is $n$-star Lindel\"{o}f. In addition, the next proposition is the Menger version of Theorem 3.1.1(3) in \cite{DRRT}.

\begin{proposition}
If $X$ is $n$-star-Menger, then $X$ is strongly $n+1$-star-Menger.
\end{proposition}

\begin{proof}
Assume $X$ is $n$-star-Menger. Let $\{\mathcal{U}_n : n \in \omega \}$ be a sequence of open covers of $X$. For each $n \in \omega$, let $\mathcal{V}_n \in [\mathcal{U}_n]^{< \omega}$ so that $\{St^n(\bigcup \mathcal{V}_n, \mathcal{U}_n ): n \in \omega\}$ is an open cover of $X$. Put $\mathcal{V}_n= \{V_n^i:i\in I_n\}$. Take for each $n \in \omega$ and $i \in I_n$, $x_n^i \in V_n^i$ and define for each $n \in \omega$, $F_n = \{x_n^i:i \in I_n\}$. Observe that for each $n \in \omega$, $\bigcup \mathcal{V}_n \subseteq St(F_n,\mathcal{U}_n)$. Hence,
$$St^n(\bigcup \mathcal{V}_n,\mathcal{U}_n) \subseteq St^n(St(F_n,\mathcal{U}_n),\mathcal{U}_n) = St^{n+1}(F_n,\mathcal{U}_n).$$
Thus, $\{St^{n+1}(F_n,\mathcal{U}_n):n\in \omega\}$ is an open cover of $X$. That is, $X$ is strongly $n+1$-star-Menger.
\end{proof}

In the class of metacompact spaces the properties $n$-star-Menger and strongly $n+1$-star-Menger coincide:

\begin{proposition}
If $X$ is strongly $n+1$-star-Menger and metacompact, then $X$ is $n$-star-Menger.
\end{proposition}

\begin{proof}
 Let $\{ \mathcal{U}_n : n \in \omega \}$ be a sequence of open covers of $X$. Since $X$ is metacompact we can assume that for each $n\in \omega$, $\mathcal{U}_n$ is point-finite.\\
 Hence, for each $n\in \omega$ there exists $F_n \in [X]^{< \omega}$ so that
 $$\big\{St^{n+1}(F_n,\mathcal{U}_n): n\in \omega\big\}$$
 is an open cover of $X$. For each $n \in \omega$, let $W_n = \{U \in \mathcal{U}_n: F_n \cap U \neq \emptyset\}$. Note $|W_n| < \omega$ and $St(F_n,\mathcal{U}_n) = \bigcup W_n$. Then for each $n \in \omega$
  $St^n\big(St(F_n,\mathcal{U}_n),\mathcal{U}_n\big) = St^n(\bigcup W_n, \mathcal{U}_n)$. Thus, $\{St^n(\bigcup W_n,\mathcal{U}_n): n\in \omega\}$ is an open cover of $X$. That is, $X$ is $n$-star-Menger.
\end{proof}

\begin{corollary}
If $X$ is strongly $2$-star-Menger and paracompact, then $X$ is Menger.
\end{corollary}

\begin{proof}
Since paracompact implies metacompact, if $X$ is strongly $2$-star-Menger and paracompact, using the previous result, $X$ is star-Menger. In addition, paracompact star-Menger spaces are Menger (see \cite{K}).
\end{proof}

Actually something stronger holds. In \cite{CGS}, the authors showed that for paraLindel\"of spaces, the three Menger-type properties, strongly star-Menger, star-Menger and Menger, are equivalent. The following result shows that this fact is still true when taking iterations of those properties.

\begin{theorem}
\label{InParaLindelofSpacesnStarMengerimplieMenger}
In the class of paraLindel\"of spaces, given any $n \in \omega $ the properties $n$-star-Menger, strongly $n$-star-Menger and Menger, are equivalent.
\end{theorem}

\begin{proof}
Fix $m \in \omega$ and let $X$ be a paraLindel\"of, $m$-star-Menger space. By Corollary \ref{paraLIndelofColapsesnStarLindelof}, $X$ is Lindel\"of and, in particular, paracompact.\\
Let us recall that a cover $\mathcal{B}$ is said to be a star refinement of a cover $\mathcal{U}$ ($\mathcal{B} \prec^* \mathcal{U}$), if for each $B \in \mathcal{B}$, there is some $U \in \mathcal{U}$ so that $St(B,\mathcal{B}) \subseteq U$. A space is paracompact if every open cover has an open star refinement (check, for instance, Theorem 5.1.12 in \cite{E}).\\
Now let $\{\mathcal{U}_n:n \in \omega\}$ be any sequence of open covers of $X$. For each $i\leq m$ define $\mathcal{B}_n^i$ open cover of $X$ so that $\mathcal{B}_n^m \prec^* \mathcal{B}_n^{m-1} \prec^* \cdots   \prec^* \mathcal{B}_n^0 \prec^* \mathcal{U}_n$.\\
Claim: For each $n\in \omega$ and for each $W\in \mathcal{B}_n^m$, there is $U_W \in \mathcal{U}_n$ so that $St^m(W,\mathcal{B}_n^m)\subseteq U_W$. Indeed, fix $n \in \omega$ and $W_m\in{\mathcal{B}_n^m}$, since $\mathcal{B}_n^m \prec^* \mathcal{B}_n^{m-1}$, there is $W_{m-1}\in \mathcal{B}_n^{m-1}$ so that $St(W_m,\mathcal{B}_n^m) \subseteq W_{m-1}$. By Lemma \ref{containmentsOfStars},  $$St^2(W_m,\mathcal{B}_n^m)= St(St(W_m,\mathcal{B}_n^m),\mathcal{B}_n^m) \subseteq St(W_{m-1},\mathcal{B}_n^{m-1}).$$
Now, since $\mathcal{B}_n^{m-1} \prec^* \mathcal{B}_n^{m-2}$, there is $W_{m-2}\in \mathcal{B}_n^{m-2}$ so that $St(W_{m-1},\mathcal{B}_n^{m-1}) \subseteq W_{m-2}$. Then, $St^2(W_m,\mathcal{B}_n^m) \subseteq W_{m-2}$. It is possible to repeat this process $m-2$ more times to get $U_{W_m} \in \mathcal{U}_n$ with $St^m(W_m,\mathcal{B}_n^m)\subseteq U_{W_m}$.\\
Since $X$ is $m$-star-Menger, let $\mathcal{W}_n \in[\mathcal{B}_n^m]^{<\omega}$ so that $\{St^m(\bigcup\mathcal{W}_n,\mathcal{B}_n^m):n\in \omega\}$ is an open cover of $X$.\\
For each $n\in \omega$ and each $W \in \mathcal{W}_n$ fix $U_W \in \mathcal{U}_n$ such that $St^m(W,\mathcal{B}_n^m)\subseteq U_W$. For all $n \in \omega$, let $\mathcal{V}_n = \{U_W: W \in \mathcal{W}_n\}$. Each $\mathcal{V}_n \in[\mathcal{U}_n]^{<\omega}$ and $\{\bigcup \mathcal{V}_n:n\in\omega\}$ is an open cover of $X$. Thus, $X$ is Menger.
\end{proof}

\section{$\Psi$-spaces}

A natural question in this context is whether every $2$-star-Menger is star-Menger. This is not the case, in fact, Tree (\cite{Tree}) built a 2-star compact (hence $2$-star-Menger), space which is not strongly 2-star Lindel\"of (in particular, not star Lindel\"of and therefore, not star-Menger).
Since, every star-Menger space is both 2-star-Menger and star-Lindel\"of, it is worth asking whether the converse holds true, i.e., Is it true that every 2-star-Menger, star-Lindel\"of space is star-Menger?. The answer is no, at least consistently (see Example \ref{2SSM, SSL,not SM} below). For this, we use a \emph{Luzin family} $\mathcal{A}$. That is, $\mathcal{A} = \{a_\alpha: \alpha < \omega_1\}$ is an almost disjoint family with the property that for each $\alpha < \omega_1$ and each $n \in \omega$ the set $\{\beta < \alpha: a_\beta \cap a_\alpha \subseteq n\}$ is finite. First of all, we introduce the analogous definitions to Definition \ref{nstarMenger} for the Rothberger property:\\

\begin{definition}\label{nstarRothberger}
\begin{enumerate}
    \item A space $X$ is called $n$-star-Rothberger if for every sequence of open covers  $\{ \mathcal{U}_n : n \in \omega \}$ there exists a sequence  $\{ U_n : n \in \omega \}$ so that for each $n \in \omega$, $U_n \in \mathcal{U}_n$ and $\{St^n(U_n,\mathcal{U}_n):n \in \omega\}$ is an open cover of $X$.
    \item  A space $X$ is called strongly $n$-star-Rothberger if for every sequence of open covers  $\{ \mathcal{U}_n : n \in \omega \}$ there exists a sequence  $\{ x_n : n \in \omega \}$ so that for each $n \in \omega$, $x_n \in X$ and $\{St^n(x_n,\mathcal{U}_n):n \in \omega\}$ is an open cover of $X$.
\end{enumerate}

Observe that $1$-star-Rothberger and strongly $1$-star-Rothberger spaces are precisely star-Rothberger and strongly star-Rothberger spaces, respectively.
\end{definition}

It is surprising that, regardless of the size of the continuum, $\Psi$-spaces induced by a Luzin family, are always strongly 2-star-Rothberger.

\begin{proposition}
\label{Luzinfam}
If $\mathcal{A}$ is Luzin, then $\Psi(\mathcal{A})$ is strongly 2-star Rothberger.
\end{proposition}

\begin{proof}
Let $\mathcal{A} = \{a_\alpha: \alpha < \omega_1\}$ be a Luzin family. Let $\{\mathcal{U}_n:n\in\omega\}$ be any sequence of open covers of $\Psi(\mathcal{A})$. For each $\alpha < \omega_1$, fix $U_\alpha^0 \in \mathcal{U}_0$ such that $a_\alpha \in U_\alpha^0$. For $\alpha < \omega_1$, let $f(a_\alpha) \in a_\alpha \cap \omega$ so that $a_\alpha \setminus f(a_\alpha) \subseteq U_\alpha^0 $.\\
For each $k \in \omega$, let $W_k = \{\alpha < \omega_1 : f(a_\alpha) = k\}$. Observe $\bigcup _{k \in \omega}W_k = \omega_1$. Furthermore, there is $A_0 \subseteq \omega_1$ stationary and $n_0 \in \omega$ such that for each $\alpha \in A_0$, $f(a_\alpha) = n_0$. Now, let $A_1 \subseteq A_0$ and $n_1 \in \omega$ so that $A_1 = \{\alpha < \omega_1 : min U_\alpha^0=n_1\}$ is stationary.
For each $k \in \omega$ and each $\alpha \in A_1$, let $B_{\alpha,k}= \{\beta < \alpha: a_\beta \cap a_\alpha \subseteq max\{k,n_0,n_1\}\}$. Since $\mathcal{A}$ is Luzin, each $B_{\alpha,k}$ is finite.\\
\textbf{Claim:} For each $k\in \omega$, $G_k = \bigcup_{\alpha \in A_1}(\alpha \setminus B_{\alpha,k})$ is cofinite.\\
Fix $k\in \omega$, there is $A^k \subseteq A_1$ stationary and $B \in [\omega]^{< \omega}$ so that for each $\alpha \in A^k$, $B_{\alpha,k} = B$. Then, $\omega_1  \setminus B = \bigcup_{\alpha \in A^k}(\alpha \setminus B) = \bigcup_{\alpha \in A^k}(\alpha \setminus B_{\alpha,k}) \subseteq \bigcup_{\alpha \in A_1}(\alpha \setminus B_{\alpha,k}) = G_k$.\\
Since for each $k \in \omega$, $G_k$ is cofinite, then $W_k \setminus G_k$ is finite. Hence, $\bigcup _{k\in \omega}W_k \cap G_k$ is cocountable.
Now we show that $\{a_\beta: \beta \in \bigcup _{k\in \omega}W_k \cap G_k\} \subseteq St^2(n_1,\mathcal{U}_0)$. Fix $k \in \omega$ and let $\beta \in W_k \cap G_k$, then $f(a_\beta) = k$ and there is $\alpha \in A^k$ such that $\beta \in \alpha \setminus B_{\alpha,k}$. Thus, $a_\beta \cap a_\alpha \not \subseteq max\{k,n_0,n_1\}$. Fix $m \in a_\beta \cap a_\alpha $ with $m> max\{k,n_0,n_1\}$. Hence, $m \in U_\alpha ^0$, $m \in U_\beta ^0$ and $m > n_1 = min  U_\alpha ^0$. Therefore, $m \in St(n_1,\mathcal{U}_0)$ and $a_\beta \in U_\beta ^0 \subseteq St(m,\mathcal{U}_0) \subseteq St^2(n_1,\mathcal{U}_0)$.

Therefore, $St^2(n_1,\mathcal{U}_0)$ contains all but countably many members of $\Psi(\mathcal{A})$, whence $\Psi(\mathcal{A})$ is strongly 2-star-Rothberger.
\end{proof}

Another fact, which is also interesting, is that $\Psi$-spaces induced by a maximal almost-disjoint family can be characterized in terms of the $strongly$ 2-$starcompact$ property (see \cite{DRRT} for information about iterated (strongly) starcompact property).

\begin{proposition}[\cite{DRRT}]
For an a.d. family $\mathcal{A}$, the following are equivalent:
\begin{enumerate}
\item $\Psi(\mathcal{A})$ is strongly 2-starcompact.
\item $\Psi(\mathcal{A})$ is strongly $k$-starcompact for every $k\geq 2$.
\item $\Psi(\mathcal{A})$ is 2-starcompact.
\item $\Psi(\mathcal{A})$ is $k$-starcompact for every $k\geq 2$.
\item $\Psi(\mathcal{A})$ is $k$-starcompact for some $k\geq 2$.
\item $\mathcal{A}$ is maximal.
\end{enumerate}
\end{proposition}

It is worth to mention that the equivalences $(3)-(6)$ were showed in \cite{DRRT} for spaces in general (by using pseudocompactness instead of a maximal almost disjoint family) and the proof for $(6)\Rightarrow(1)$ is contained in the proof of Example 2.2.5 in same article. For convenience of the reader, we outline the proof of these equivalences for $\Psi$-spaces below.

\begin{proof}
Since the implications $(1)\Rightarrow(2)\Rightarrow(3)\Rightarrow(4)\Rightarrow(5)$ always hold, we just show $(5)\Rightarrow(6)$ and $(6)\Rightarrow(1)$.

Assume that $\Psi(\mathcal{A})$ is $k$-starcompact for some $k\geq 2$. If $\mathcal{A}$ is not maximal, there is some $B\subseteq \omega$ which is almost disjoint from every member of $\mathcal{A}$. Consider the open cover $\mathcal{U}$ consisting of singletons from $\omega$ together with sets of the form $A\setminus B$ for $A\in\mathcal{A}$. Then for any $\mathcal{W}\subseteq \mathcal{U}$, $St^k(\bigcup \mathcal{W}, \mathcal{U})$ intersects $B$ just in those natural numbers whose singletons were chosen in $\mathcal{W}$, as all other members of $\mathcal{U}$ are disjoint from $B$. Therefore $\Psi(A)$ is not $k$-starcompact.

If $\mathcal{A}$ is not 2-strongly starcompact, then there is an open cover $\mathcal{U}$ so that the star of every finite subcollection misses some element of $\omega$, otherwise taking another star would give all of $\Psi(\mathcal{A})$. Now we can choose an infinite subset $B$ with increasing enumeration $\{b_i:i<\omega\}$ of $\omega$ recursively so that for every $j<\omega$, $b_j\not\in St(\{b_i:i<j\}, \mathcal{U})$. Now we claim that $B$ is almost disjoint from every member of $\mathcal{A}$, and so $\mathcal{A}$ cannot be maximal. To see this, suppose that there is $A\in \mathcal{A}$ which intersects $B$ infinitely. Then the neighborhood of $A$ in $\mathcal{U}$ contains two points $b_{j_0}, b_{j_1}$ with $j_0<j_1$. So this neighborhood is contained in $St(b_{j_0},\mathcal{U})$, contradicting the choice of $j_1$. 
\end{proof}

\begin{example}
\label{2SSM, SSL,not SM}
$(\mathfrak{d} = \omega_1)$ There is a Tychonoff, strongly 2-star-Rothberger, strongly star-Lindel\"of, not star-Menger space.
\end{example}

\begin{proof}
Let $\mathcal{A}$ be a Luzin family. Since $\Psi(\mathcal{A})$ is separable, it is strongly star-Lindel\"of and, by Proposition \ref{Luzinfam}, it is strongly 2-star-Rothberger. 
Since $\mathfrak{d}=\omega_1$, by Proposition 9 in \cite{BM},  $\Psi(\mathcal{A})$ is not star-Menger.
\end{proof} 

\begin{corollary}
If $\mathfrak{d} = \omega_1$ then there is a Tychonoff, 2-star Menger, star-Lindel\"of, not star-Menger space.
\end{corollary}

In contrast to $\Psi$-spaces induced from Luzin families, we also have examples of $\Psi$-spaces that not hold any iteration of star-Menger property.

\begin{example}\label{noparM}
There is a $\Psi$-space which is not $m$-star Menger for any $m<\omega$.
\end{example}
\begin{proof}
Identify $\omega^{<\omega}$ with $\omega$ and let $\mathcal{A}$ be the branches of the tree $\omega^{<\omega}$. Now, for each $n$, let $\mathcal{U}_n$ be the open cover consisting of the sets $N_s=\{a\in \Psi(\mathcal{A}):s\sqsubseteq a\}$, where $s\in\omega^{n+1}$, and $\{x\}$ for all $x\not\in\bigcup\{N_s:s\in\omega^{n+1}\}$. Then, for each $n\in\omega$, $\mathcal{U}_n$ consists of pairwise disjoint clopen sets and therefore, no elements are picked up when taking stars.

Suppose $\mathcal{V}_n\subseteq \mathcal{U}_n$ is finite. By induction, define $x\in\mathcal{A}$ so that $N_{x\upharpoonright (n+1)}\not\in \mathcal{V}_n$. It is clear that $x\not\in St^m(\bigcup\mathcal{V}_n,\mathcal{U}_n)$ for any $m,n<\omega$.
\end{proof}

We can also define $\Psi$-spaces associated to maximal almost disjoint families that do not satisfy any iteration of the star-Rothberger property.

\begin{example}
There is a MAD family $\mathcal{A}$ so that $\Psi(\mathcal{A})$ is not $k$-star Rothberger for any $k$.
\end{example}
\begin{proof}
Enumerate all infinite subsets of $\omega$ as $\langle b_\alpha:\alpha<\mathfrak{c}, \alpha \textrm{ even} \rangle$ and all functions in $\prod_n {}^n3$ as $\langle s_\alpha:\alpha<\mathfrak{c}, \alpha \textrm{ odd} \rangle$. Note that any ordinal can be written as $\alpha+n$ for some $n<\omega$, and ``odd'' and ``even'' above refer to the parity of this integer $n$.

Let $\{B_t:t\in {}^{<\omega}3\}$ be a sequence of infinite subsets of $\omega$ so that $B_\emptyset=\omega$ and for all $t\in {}^{<\omega}2$, $\{B_{t^\frown 0},B_{t^\frown 1}, B_{t^\frown 2}\}$ partition $B_t$. 

We will define $\mathcal{A}=\{a_\alpha:\alpha<\mathfrak{c}\}$. Suppose $a_\beta$, $\beta<\alpha$, have already been defined.

We will construct $a_\alpha$ infinite subsets of $\omega$ and $x_\alpha$ branches of the tree so that
\begin{enumerate}
    \item For $\beta<\alpha$, $a_\alpha$ is almost disjoint from $a_\beta$.
    \item For each $n$, $\{z_m:m\ge n\} \subseteq B_{x_\alpha\upharpoonright n}$, where $\{z_m:m<\omega\}$ is the increasing enumeration of $a_\alpha$ (so that in particular, $a_\alpha\subseteq^* B_{x_\alpha\upharpoonright n}$ for any $n$).
\end{enumerate}

If $\alpha$ is even, then we will construct $a_\alpha$ so that

\begin{enumerate}
    \item[3'.] If $b_\alpha$ is almost disjoint from each $a_\beta,$ $\beta<\alpha$, then $a_\alpha\subseteq b_\alpha$.
\end{enumerate}
Let us assume the case where $b_\alpha$ is almost disjoint from each $a_\beta,$ $\beta<\alpha$, and hence the third condition applies. Note that if $3'$ is satisfied, then so is the first condition. Let $x_\alpha$ be a branch through the tree so that $B_{x_\alpha\upharpoonright n}\cap b_\beta$ is infinite for each $n$. Define $\{z_n:n<\omega\}$ to be an increasing sequence of natural numbers so that $z_n\in B_{x_\alpha\upharpoonright n}\cap b_\alpha$. 

If $\alpha$ is odd, then we will construct $a_\alpha$ so that
\begin{enumerate}
    \item[3''.] For each $n$, $\{z_m:m\ge n\}\cap B_{s_\alpha(n+1)}=\emptyset$.
\end{enumerate}
This construction will proceed in $\omega$ steps. Let $S$ be the subtree of ${}^{<\omega}3$ of all $t$ so that for all $s\le t$, $s \neq s_\alpha(n)$, where $n$ is the length of $s$. The resulting subtree $S$ still splits at every node.

Let $x_\alpha$ be a branch through this tree which is not $x_\beta$ for any $\beta<\alpha$. This is possible since the tree has $\mathfrak{c}$ branches.

In step $n$, pick $z_n\in B_{x_\alpha\upharpoonright (n+1)}$ greater than all previously chosen $z_m$, $m<n$. Let $a_\alpha = \{z_n:n<\omega\}$. As $B_{x_\alpha\upharpoonright n}$ is $\subseteq$-decreasing along the branch, this ensures that for each $n$, $\{z_m:m\ge n\} \subseteq B_{x_\alpha\upharpoonright n}$. 

If $\beta<\alpha$, then let $i$ be so that $x_\alpha(i)\neq x_\beta(i)$. We have that $a_\alpha\subseteq^* B_{x_\alpha\upharpoonright(i+1)}$ and $a_\beta\subseteq^* B_{x_\beta\upharpoonright(i+1)}$, yet $B_{x_\alpha\upharpoonright(i+1)}$ is disjoint from $B_{x_\beta\upharpoonright(i+1)}$. Therefore $a_\alpha$ is almost disjoint from $a_\beta$.

Since $x_\alpha\upharpoonright (n+1)\neq s_\alpha(n+1)$ and $\{z_m:m\ge n\} \subseteq B_{x_\alpha\upharpoonright n}$, we have that 3'' holds.

For each $n$, let $a_\alpha^{(n)}$ be the set $a_\alpha$ with the least $n$ elements removed and let $\mathcal{U}_n$ be the open cover of $\Psi(\mathcal{A})$ consisting of singletons from $\omega$ together with $\{a_\alpha\}\cup a_\alpha^{(n)}$. This sequence of covers shows that $\Psi(\mathcal{A})$ is not $k$-star Rothberger. For any selection of sets $\{U_n:n<\omega\}$, for each $n$ there is some $s(n)\in{}^n3$ so that $U_n\subseteq B_{s(n)}$. Now take $\alpha$ so that $s=s_\alpha$. Then $a_\alpha$ is disjoint from the $k$-iterated star of $U_n$ in $\mathcal{U}_n$ for each $n$ by construction. 
\end{proof}

We finish this section giving a result on $\Psi$-spaces that combines the style of the results in Section \ref{sectionsmallunions} with iterations of a star selection property introduced in Section \ref{sectioniteratedstars}.

\begin{theorem}
For any $k\in\omega$, if $X$ is the union of less than $\mathfrak{b}$ k-strongly star-Menger $\Psi$-spaces on $\omega$, then $X$ is k-strongly star-Menger.
\end{theorem}

\begin{proof}
Fix $k\in\omega$. Let $\kappa$ be a cardinal less than $\mathfrak{b}$. For each $\alpha<\kappa$, let $\Psi_{\alpha}$ be a $\Psi$-space on $\omega$ defined by an a.d. family $\mathcal{A}_\alpha$ and put $X=\bigcup_{\alpha<\kappa}\Psi_\alpha$ where each $\Psi_\alpha$ is $k$-strongly star-Menger. Let $\{ \mathcal{U}_n : n \in \omega \}$ be a sequence of open covers of $X$ consisting of basic open sets (for each $\alpha<\kappa$, $n\in\omega$ and $U\in\mathcal{U}_n$, $|U\cap \mathcal{A}_\alpha|\leq 1$). For each $\alpha<\kappa$, let $F_n^\alpha\in[\Psi_\alpha]^{<\omega}$ such that for every $m\in\omega$,  $\{St^k(F_n^\alpha, \mathcal{U}_n): m \leq n < \omega\}$ is an open cover of $\Psi_\alpha$. 

Fix $\alpha<\kappa$ and $n\in\omega$. For each $A\in F_n^\alpha\cap\mathcal{A}_\alpha$, let $U_A$ be a member of $\mathcal{U}_n$ such that $A\in U_A$. So, $U_A=\{A\}\cup A \setminus F_A$ for some $F_A\in[A]^{<\omega}$. For each $A\in F_n^\alpha\cap\mathcal{A}_\alpha$, fix $n_A\in A\setminus F_A$. Thus, for each $\alpha<\kappa$ and for each $n\in\omega$, let $$G_n^\alpha=(F_n^\alpha\cap\omega)\cup\{n_A:A\in F_n^\alpha\cap\mathcal{A}_\alpha\}\cup\{\bigcup F_A:A\in F_n^\alpha\cap\mathcal{A}_\alpha\}.$$
Then, for each $\alpha<\kappa$ and each $n\in\omega$, we have $G_n^\alpha\in[\omega]^{<\omega}$ and $St(F_n^\alpha,\mathcal{U}_n)\subseteq St(G_n^\alpha,\mathcal{U}_n)$. Indeed, let $x\in St(F_n^\alpha,\mathcal{U}_n)$. Then there exists $U\in \mathcal{U}_n$ such that $x\in U$ and $U\cap F_n^\alpha\neq\emptyset$. We have two cases:\\
If $U\cap(F_n^\alpha\cap\omega)\neq\emptyset$, then $U\cap G_n^\alpha\neq\emptyset$ and therefore, $x\in U\subseteq St(G_n^\alpha,\mathcal{U}_n)$. \\
If $U\cap(F_n^\alpha\cap\mathcal{A}_\alpha)\neq\emptyset$, then $U=\{A\}\cup A\setminus F$ for some $A\in F_n^\alpha\cap \mathcal{A}_\alpha$ and for some $F\in[A]^{<\omega}$. Then $x\in St(F_A\cup\{n_A\},\mathcal{U}_n)\subseteq St(G_n^\alpha,\mathcal{U}_n)$.

We conclude that $St(F_n^\alpha,\mathcal{U}_n)\subseteq St(G_n^\alpha,\mathcal{U}_n)$.

Now, we define, for each $\alpha<\kappa$, a function $f_\alpha:\omega\rightarrow\omega$ as $f_\alpha(n)=max (G_n^\alpha)$ for each $n\in\omega$. Since the collection $\{f_\alpha: \alpha < \kappa \}$ has size less than $\mathfrak{b}$, there exists $g \in \omega^\omega$ such that for every $\alpha < \kappa$, $f_\alpha \leq ^* g$. For each $n \in \omega$, let $D_n =\{i\in\omega: 0\leq i \leq g(n) \}$. Then each $D_n$ is a finite subset of $\omega$ and it follows that $\{St^k(D_n,\mathcal{U}_n): n\in\omega\}$ is an open cover of $X$. Indeed, let $x \in X$. Then, there exists $\alpha < \kappa$ such that $x\in \Psi_\alpha$. Since $f_\alpha \leq ^* g$, there is $m\in\omega$ so that for every $n\geq m$, $f_\alpha(n)\leq g(n)$. Furthermore, since the collection $\{St^k(F_n^\alpha, \mathcal{U}_n): m \leq n < \omega\}$ is an open cover of $\Psi_\alpha$, let $n \geq m$ such that $x\in St^k(F_n^\alpha, \mathcal{U}_n)$. We obtain that $St(F_n^\alpha, \mathcal{U}_n) \subseteq St(G_n^\alpha, \mathcal{U}_n) \subseteq St(\{1,\ldots, f_\alpha(n)\}, \mathcal{U}_n) \subseteq St(\{1,\ldots, g(n)\}, \mathcal{U}_n)= St(D_n, \mathcal{U}_n)$. Thus, $x\in St^k(F_n^\alpha,\mathcal{U}_n)\subseteq St^k(D_n,\mathcal{U}_n)$. Therefore, the collection $\{St^k(D_n,\mathcal{U}_n):n\in\omega\}$ is an open cover of $X$. Thus, $X$ is $k$-strongly star-Menger.
\end{proof}

\section{Normal star-Menger not strongly star-Menger not Dowker space}

\noindent Recall that $X$ is a Dowker space if and only if $X$ is normal and its Cartesian product with the closed unit interval $I$ is not normal. Equivalently,  $X$ is normal and not countably paracompact. In \cite{CGS} the following questions were posed:

\begin{question}[\cite{CGS} Question 2.4]
\label{Normal SM not SSM}
Is there a normal star-Menger space which is not strongly star-Menger?
\end{question}

\begin{question}[\cite{CGS} Question 2.21]
\label{Normal SM not SSM is it Dowker}
Are normal, countably paracompact star-Menger spaces strongly star-Menger? I.e., if $X$ is normal, star-Menger, not strongly star-Menger, is $X$ a Dowker Space?
\end{question}

\noindent In this section we present a consistent example (Example \ref{normalSLnotSSL} below), of a normal star Menger not strongly star-Menger not Dowker space. This space answers consistently in the afirmative Question \ref{Normal SM not SSM} and in the negative Question \ref{Normal SM not SSM is it Dowker}.

\noindent In \cite{Ta84} Tall presented an example of a separable normal space with an uncountable discrete subspace. Below we provide details of the construction of such example for sake of completeness:

\begin{example}[\cite{Ta84} Example E]
\label{Tallexample}
Assuming $2^{\aleph_0} = 2^{\aleph_1}$ there exists a separable normal $T_1$ space with an uncountable closed subspace.
\end{example}

\noindent {\bf Construction:} Let $L$ be a set of cardinality $\aleph_1$ disjoint from $\omega$. The existence of a strongly independent family $\mathcal{F}$\footnote{For an infinite cardinal $\kappa$, a family $\mathcal{F} \subseteq \mathcal{P}(\kappa)$ is called \emph{independent} if for all pairs of disjoint $F,G \in [\mathcal{F}]^{<\omega}$ we have:
$$C_{F,G} = \bigcap_{A\in F}A \cap \bigcap _{A \in G}(\kappa \smallsetminus A) \neq \emptyset$$
(Assume $\bigcap \emptyset = \kappa$). If in addition, for each pair $(F,G)$ as above, $|C_{F,G}| = \kappa$, $\mathcal{F} $ is called \emph{strongly independent}.} of subsets of $\aleph_0$ of size $2^{\aleph_0} = \mathfrak{c}$ is guaranteed by the Fichtenholz-Kantorovitch-Hausdorff Theorem\footnote{For every infinite cardinal $\kappa$ there exists a strongly independent family $\mathcal{F} \subseteq \mathcal{P}(\kappa)$ such that $|\mathcal{F}| = 2^\kappa$.}.

Write $\mathcal{F} = \{A_\alpha: \alpha < \mathfrak{c}\}$. Since $|L| =\aleph_1$, $|\mathcal{P}(L)| = 2^{\aleph_1}$. Assuming $2^{\aleph_0} = 2^{\aleph_1}$ it is possible to build a function $f:\mathcal{P}(L) \to \{A_\alpha: \alpha < \mathfrak{c}\} \cup \{\omega \smallsetminus A_\alpha: \alpha < \mathfrak{c}\}$ which is bijective and complement-preserving (for each $B \subseteq L$, $f(L\smallsetminus B) = \omega \smallsetminus f(B)$).\\
Now let $X = L \cup \omega$ with a subbase $\varphi$ for a topology defined by
\begin{enumerate}
\item if $M\subseteq L$, then $M \cup f(M) \in \varphi$,
\item if $n \in \omega$, then $\{n\} \in \varphi$,
\item if $p \in X$, then $X \smallsetminus \{p\} \in \varphi$.
\end{enumerate}
Observe that by condition (3) $X$ is $T_1$. By (2) $\omega$ is open, therefore $L = X \smallsetminus \omega$ is closed and, by (1) for any $x\in L$, $\{x\} \cup f(\{x\})$ is open such that $[\{x\} \cup f(\{x\})] \cap L = \{x\}$, that is $L$ is discrete. $X$ is separable since $\omega$ is dense in $X$: let $U$ be any nonempty  basic open set, then 
$$U = \bigcap _{U\in F} U \cap \bigcap _{U\in G} U \cap \bigcap _{U\in H} U$$
where $F, G, H$ are finite (possibly empty), each $U \in F$ is a subbasic open set defined as in (1), each $U \in G$ is a subbasic open set defined as in (2), and each $U \in H$ is a subbasic open set defined as in (3). To show $U \cap \omega \neq \emptyset$ it is enough to observe that $| \big(\bigcap _{U\in F} U\big) \cap \omega| = \omega$. This is always the case since $\mathcal{F}$ is a strongly independent family. Now let $Y,Z$ be disjoint closed subsets of $X$ and observe:

\begin{eqnarray*}
U_Y & = &\big( (Y \smallsetminus L) \cup [(Y \cap L) \cup f(Y \cap L)] \big) \cap (X \smallsetminus Z) \\
& = & \big(Y \cup f(Y \cap L)\big) \cap (X \smallsetminus Z) \\
U_Z & = & \big( (Z \smallsetminus L) \cup [(L \smallsetminus Y) \cup f(L \smallsetminus Y)] \big) \cap (X \smallsetminus Y) 
\end{eqnarray*}
are open sets and $Y \subseteq  U_Y$, $Z \subseteq U_Z$. Assume $x \in U_Y \cap U_Z$, then $x \in X \smallsetminus (Y \cup Z)$ and $x \in f(Y\cap L) \cap f(L \smallsetminus Y)$. But this is a contradiction since $f$ is complement preserving: $f(L \smallsetminus Y) = f(L \smallsetminus (Y\cap L)) = \omega \smallsetminus (Y \cap L)$. Hence, $X$ is normal.
$\blacksquare$\\

\noindent The following example presented by Song in \cite{YKS1102} and \cite{YKS1208} is a modification of Example \ref{Tallexample}. Song proved, in particular, that this space is normal, star-Lindel\"of and not strongly star-Lindel\"of (actually he showed something stronger: there is $\mathcal{U} \in \mathcal{O}(X)$ such that for all $L \subseteq X$ Lindel\"of subspace of $X$, $St(L,\mathcal{U}) \neq X$).

\begin{example}[\cite{YKS1102}, \cite{YKS1208}]
\label{normalSLnotSSL} 
Assuming $2^{\aleph_0} = 2^{\aleph_1}$ there exists a normal $T_1$ space which is star-Lindel\"of and not strongly star Lindel\"of.
\end{example}

\noindent {\bf Construction:}  Let $X_0 = L \cup \omega$ denote the space built in Example \ref{Tallexample}. Let $X = L \cup (\omega_1 \times \omega)$ and topologize it as follows, a basic open set of
\begin{description}
\item[(i)] $x \in L$ is a set of the form $V^U_\alpha(x) = (U \cap L) \cup \big((\alpha, \omega_1) \times (U \cap \omega)\big)$ where $U$ is a neighbourhood of $x \in X_0$ and $\alpha < \omega_1$.
\item[(ii)] $\langle \alpha , n \rangle \in (\omega_1 \times \omega)$ is a set of the form $V_W(\langle \alpha , n \rangle) = W \times \{n\}$ where $W$ is a neighbourhood of $\alpha$ in $\omega_1$ with the usual topology.
\end{description}

Condition (i) guarantees that $X$ is $T_1$. Furthermore, $\omega_1 \times \omega$ is open in $X$ and for $x\in L$, if we let $U = \{x \} \cup f(\{x\})$. then for any $\alpha < \omega_1$, $V_\alpha ^U(x) \cap L = \{x\}$. That is, $L$ is closed and discrete in $X$.\\

\noindent {\bf $X$ is normal:} Let $Y, Z \subseteq  X$ closed and disjoint. Define $Y_L = Y \cap L$ and $Z_L = Z \cap L$ and for each $n \in \omega$, $Y_n = Y \cap (\omega_1 \times \{n\})$, $Z_n = Z \cap (\omega_1 \times \{n\})$. Since $Y \cap Z = \emptyset$ and $\omega_1 \times \{n\}$ is a copy of $\omega_1$ with the usual topology (for each $n \in \omega$), then we can find clopen sets $Y_n', Z_n' \subseteq \omega_1 \times \{n\}$ such that $Y_n' \cap Z_n' = \emptyset$, $Y_n \subseteq Y_n' $, $Z_n \subseteq Z_n'$ and so that for each $n\in \omega$, $Y_n'$ is cofinal in  $\omega_1 \times \{n\}$ if and only if $Y_n$ is cofinal in  $\omega_1 \times \{n\}$ and $Z_n'$ is cofinal in  $\omega_1 \times \{n\}$ if and only if $Z_n$ is cofinal in  $\omega_1 \times \{n\}$. This is possible since for each $n\in \omega$, $Y_n$ and $Z_n$ cannot be both cofinal (otherwise $Y_n \cap Z_n \neq \emptyset$). Let
$$\mathcal{Y} = Y_L \cup \bigcup _{n\in \omega} Y_n', \qquad \mathcal{Z} = Z_L \cup \bigcup _{n\in \omega} Z_n'$$
Observe $Y \subseteq \mathcal{Y}$, $Z \subseteq \mathcal{Z}$ and $\mathcal{Y} \cap \mathcal{Z} = \emptyset$.\\
\emph{Claim:} $\mathcal{Y}$ and $\mathcal{Z}$ are closed in $X$.\\
Indeed, if $\langle\alpha , m \rangle \in (\omega_1 \times \omega ) \smallsetminus \mathcal{Y}$, since $Y_m'$ is clopen in $\omega_1 \times \{m\}$, then there is $U$ open neighbourhood of $\langle\alpha , m \rangle$ in $\omega_1 \times \{m\}$ (and therefore open neighbourhood in $X$), such that $U \cap Y_m' = \emptyset$. Now, let $x \in L \smallsetminus \mathcal{Y}$ and assume that for each $U$ open neighbourhood of $x$ in $X_0$ and each $\alpha < \omega_1$, $V_\alpha ^U (x) \cap \mathcal{Y} \neq \emptyset$. This implies that for each $U$ open neighbourhood of $x$ in $X_0$ and each $\alpha < \omega_1$ there is some $n\in \omega$ such that $V_\alpha ^U(x) Y_n' \neq \emptyset$ and $Y_n'$ is cofinal in $\omega_1 \times \{n\}$. Then $Y_n$ is cofinal in $\omega_1 \times \{n\}$ and $V_\alpha ^U(x) Y_n \neq \emptyset$. Hence, $x \in \overline{Y} = Y$ which is a contradition. Thus, $\mathcal{Y}$ is closed. A similar argument shows that $\mathcal{Z}$ is closed.\\
	Since $Y_L$ and $Z_L$ are disjoint closed subsets of $X_0$ and $X_0$ is normal (recall $X_0$ is the space constructed in Example \ref{Tallexample}), then there exist disjoint open sets $U_Y$, $U_Z$ in $X_0$ such that $Y_L \subseteq U_Y$, $Z_L \subseteq U_Z$. Let
$$V_Y = (U_Y \cap Y) \cup \bigcup _{n \in U_Y \cap \omega} (\omega_1 \times \{n\}, \qquad V_Z = (U_Z \cap Z) \cup \bigcup _{n \in U_Z \cap \omega} (\omega_1 \times \{n\}.$$
Observe that $V_Y$ and $V_Z$ are disjoint open subsets in $X$ and $Y_L \subseteq V_Y$, $Z_L \subseteq V_Z$. Let $W_Y = \mathcal{Y} \cup (V_Y \smallsetminus \mathcal{Z})$, $W_Z = \mathcal{Z} \cup (V_Z \smallsetminus \mathcal{Y})$. Hence, $W_Y$ and $W_Z$ are open sets in $X$, $W_Y \cap W_Z = \emptyset$, and $y\ \subseteq W_Y$, $Z\subseteq W_Z$.\\

\noindent {\bf $X$ is not strongly star-Lindel\"of:} List $L = \{x_\alpha: \alpha < \omega_1\}$. Since $L$ is a closed discrete subset of $X_0$, for $\alpha < \omega_1$ let $D_\alpha$ be an open neighbourhood of $x_\alpha$ in $X_0$ such that $D_\alpha \cap L = \{x_\alpha\}$. Hence, 
$$\mathcal{U} = \{V_\alpha ^{D_\alpha}(x_\alpha): \alpha < \omega_1\} \cup \{\omega_1 \times \omega\} \in \mathcal{O}(X).$$
Assume $E \in [X]^{\omega}$, we show $St(E,\mathcal{U}) \neq X$. Since $E$ is countable, fix $\beta_0 , \beta_1 <\omega_1$ such that $sup\{\alpha: x_\alpha \in E \cap L\} <\beta_0$ and $sup\{\gamma:\langle \gamma,n\rangle \in E \textit{ for some } n \in \omega\} < \beta_1$. Let $\alpha = max\{ \beta_0 , \beta_1\}$ and observe $E \cap V_\alpha^{D_\alpha}(x_\alpha) = \emptyset$ Since $V_\alpha^{D_\alpha}(x_\alpha)$ is the only element of $\mathcal{U}$ that contains $x_\alpha$, then $x_\alpha \notin St(E, \mathcal{U})$. Thus, $X$ is not strongly star-Lindel\"of.\\

\noindent {\bf $X$ is star-Lindel\"of:} Let $\mathcal{U} \in \mathcal{O}(X)$ and define
$$M = \{n \in \omega: (\exists U \in \mathcal{U})(\exists \beta < \omega_1)[(\beta,\omega_1) \times \{n\} \subseteq U]\}.$$
For each $n \in M$ fix $U_n \in \mathcal{U}$ and $\beta_n < \omega_1$ such that $(\beta,\omega_1) \times \{n\} \subseteq U_n$. Put $\mathcal{V}' = \{U_n:n \in M\}$.\\
\emph{Claim:} $L \subseteq St(\bigcup \mathcal{V}',\mathcal{U})$.\\
Indeed, let $x \in L$, there is $U^x \in \mathcal{U}_n$ such that $x \in U^x$ and therefore, there is $U$ open neighbourhood of $x$ in $X_0$ and $\alpha < \omega_1$ such that $V_\alpha^U(x) \subseteq U^x$. Since $V_\alpha^U(x) \cap (\omega_1 \times \omega) = (\alpha , \omega_1) \times (U \cap \omega)$ and $U = N \cup f(N)
$ for some $N \subseteq L$, with $x \in N$, it holds true that $n \in f(N) \rightarrow n \in M$. Then, for $n \in f(N)$, $V_\alpha^U(x) \cap (\omega_1 \times \{n\}) \cap [(\beta_n,\omega_1)\times \{n\}] \neq \emptyset$. Thus, $V_\alpha^U(x) \cap U_n \neq \emptyset$. Hence, $U^x \cap U_n \neq \emptyset$. Therefore $x \in St(U_n, \mathcal{U}) \subseteq St(\bigcup \mathcal{V}',\mathcal{U})$. Now, $\omega_1 \times \omega$ is a countable union of strongly star compact spaces, then there is a countable $\mathcal{V}'' \subseteq \mathcal{U}$ such that $\omega_1 \times \omega \subseteq St(\bigcup \mathcal{V}'',\mathcal{U})$. If we let $\mathcal{V} = \mathcal{V}' \cup \mathcal{V}''$, then $St(\bigcup \mathcal{V},\mathcal{U}) = X$. $\blacksquare$\\

\begin{proposition}
\label{normalSMnotSSM} 
Assuming $2^{\aleph_0} = 2^{\aleph_1}$ and $\aleph_1 < \mathfrak{d}$ the space $X$ built in Example \ref{normalSLnotSSL} is normal, star-Menger, and is not either strongly star-Menger nor Dowker.
\end{proposition}

\begin{proof}
It has been shown that $X$ is normal and not strongly star-Lindel\"of (in particular, $X$ is not strongly star-Menger). It remains to show that it is star-Menger and is not a Dowker space.\\

\noindent {\bf $X$ is star-Menger}: let $(\mathcal{U}_n: n \in \omega )$ be any sequence of open covers of $X$. Write $L = \{x_\alpha: \alpha < \omega_1\}$ and for each $\alpha < \omega_1$ and each $n \in \omega$, let $f_\alpha(n) = min\{i \in \omega: (\exists U \in \mathcal{U}_n)(\exists \beta < \omega_1)[x_\alpha \in U \wedge (\beta,\omega_1) \times \{i\} \subseteq U]\}$. Observe that for each $\alpha < \omega_1$, $f_\alpha : \omega \to \omega$ is well defined. Since $\{f_\alpha: \alpha < \omega_1\}$ has size less than $\mathfrak{d}$, there is a funtion $g \in \omega^\omega$ such that for all $\alpha < \omega_1: g \not \le ^* f_\alpha$. For $n \in \omega$ let 
$$M_n = \{i \in \omega: (\exists U \in \mathcal{U}_n)(\exists \beta < \omega_1)[(\beta,\omega_1) \times \{i\} \subseteq U]\}.$$
Now, for each $n \in \omega$ and each $i \in M_n$, fix $U_n ^i \in \mathcal{U}_n$ and $\beta_n ^i < \omega_1$ such that $(\beta _n ^i,\omega_1) \times \{i\} \subseteq U_n^i$ and let $\mathcal{V}_n = \{U_n^i : i \in M_n \cap g(n)\}$.\\
\emph{Claim:} $L \subseteq \bigcup \{St(\bigcup \mathcal{V}_n,\mathcal{U}_n):n \in \omega\}$.\\
Indeed, fix $x_\alpha \in L$. There is $n \in \omega$ such that $f_\alpha (n) < g(n)$. Hence, there are $U \in \mathcal{U}_n$ and $\beta < \omega_1$ such that $x_\alpha \in U$ and $(\beta,\omega_1) \times \{f_\alpha(n)\} \subseteq U$. Thus, $f_\alpha(n) \in M_n$ and $U_n ^{f_\alpha(n)} \in \mathcal{V}_n$. In addition, $U_n^{f_\alpha(n)} \cap U \neq \emptyset$. Hence, $x \in St(\bigcup\mathcal{V}_n,\mathcal{U}_n) \subseteq \bigcup \{St(\bigcup \mathcal{V}_n,\mathcal{U}_n):n \in \omega\}$.\\

\noindent {\bf $X$ it is not a Dowker space}: Let us recall the following characterization: A normal space $D$ is a Dowker space (see \cite{MER}) if, and only if, $D$ has a countable increasing open cover $\{U_n:n\in \omega\}$ such that there is no closed cover $\{F_n:n\in \omega\}$ of $D$ with $F_n \subseteq U_n$ for each $n\in \omega$. Hence, let $\{U_n:n\in \omega\}$ be any countable increasing open cover ($U_0 \subseteq U_1 \subseteq \cdots$) of $X$, we must find a countable cover of closed sets $\{F_n:n\in \omega\}$, such that for each $n\in \omega$, $F_n \subseteq U_n$.\\

\noindent For each $i\in \omega$ define $n_i = min\{n\in\omega:i\leq n \wedge (\exists \gamma < \omega_1)\big[ [\gamma,\omega_1) \times \{i\} \subseteq U_n\big]\}$. Observe that since $\{U_n:n\in \omega\}$ is a countable cover of $X$, $n_i$ is well defined for each $i\in \omega$. In addition, for each $n\in \omega$ and $i\in \omega$ with $i \le n_i\le n$ let 
$$\gamma_i^n = min\{\gamma < \omega_1: [\gamma,\omega_1) \times \{i\} \subseteq U_n\} \qquad (*)$$
Since for each $n\in \omega$, $U_n \subseteq U_{n+1}$, then $\gamma_i^n$ is well defined. Now, for $n\in \omega$ let $$F_n = \big( \bigcup _{i \le n} \{[\gamma_i^n,\omega_1) \times \{i\}: i\le n_i \le n\} \big) \cup (U_n \cap L).$$\\
\emph{Claim:} 
\begin{description}
\item[(1)] For each $n\in \omega$, $F_n$ is closed,
\item[(2)] For each $n\in \omega$, $F_n \subseteq U_n$,
\item[(3)] $\bigcup _{n\in \omega}F_n = X$.
\end{description}
Indeed, to show (1), fix $n\in \omega$. First assume $x \in (X \smallsetminus F_n) \cap (\omega_1 \times \omega)$. Hence $x = \langle \alpha , m \rangle$ for some $\alpha < \omega_1$ and $m \in \omega$. If $F_n \cap (\omega_1 \times \{m\}) = \emptyset$, any $U \subseteq \omega_1 \times \{m\}$ open neighbourhood of $x$ is disjoint from $F_n$. If $F_n \cap (\omega_1 \times \{m\}) \neq \emptyset$, then $\alpha < \gamma _m^n$ and for each $\beta < \alpha$, $(\beta,\alpha] \times \{m\}$ is an open neighbourhood of $x$ disjoint from $F_n$. Now, assume $x \in (X \smallsetminus F_n) \cap L$, let $N \subseteq L$ such that $N \cap F_n = \emptyset$ and $x \in N$. Observe that $U = N \cup f(N) \smallsetminus (n+1) = \big(N \cup f(N)\big) \cap \big(\bigcap_{j \le n+1}(X_0 \smallsetminus \{j\})\big)$ is an open neighbourhood of $x$ in $X_0$ (see condition (1) and (3) of Example  \ref{Tallexample}). Hence, for any $\alpha < \omega_1$, $V_\alpha ^U(x)$ ( $= [U \cap L ] \cup [(\alpha,\omega_1) \times (U \cap \omega)]$) is an open neighbourhood of $x$ in $X$ such that $V_\alpha^U(x) \cap F_n = \emptyset$ since $F_n \subseteq \omega_1 \times [0,n]$ and $V_\alpha^U(x) \cap (\omega_1 \times [0,n]) = \emptyset$. Thus, $F_n$ is closed.\\
To show (2), fix $n \in \omega$. If $x \in F_n \cap L$, then $x \in U_n$. If $x = \langle \alpha , m \rangle \in F_n \cap (\omega_1 \times \omega)$, then there is some $i \le n_i \le n$ such that $\langle \alpha, m \rangle \in [\gamma_i ^n,\omega_1) \times \{i\}$. Thus , $m = i$ and $[\gamma_i ^n,\omega_1) \times \{i\} \subseteq U_n$. Hence $F_n \subseteq U_n$.\\
Let us show (3). If $x \in X \cap L$, then there is some $n\in \omega$ such that $x\in U_n$. Hence, $x\in U_n \cap L \subseteq F_n$. If $x\in X \smallsetminus L$, there is some $i \in \omega$ such that $x \in \omega_1 \times \{i\}$. By $(*)$ and the fact that $U_n \subseteq U_{n+1}$, $\{\gamma_i^n:n\in \omega\}$ is a decreasing sequence of ordinals. Since $U_n:n\in\omega)$ covers $X$, there is some $m\in \omega$ such that $\gamma_i^m = 0$. Thus, $x \in F_m$.
\end{proof}

\vspace{.5cm}

\textsc{Department of Mathematics and Statistics, York University, 4700 Keele St. Toronto, ON M3J 1P3 Canada}\par\nopagebreak

\vspace{.1cm}
\textit{Email address}: J. Casas-de la Rosa: \texttt{olimpico.25@hotmail.com}

\vspace{.1cm}
\textit{Email address}: W. Chen-Mertens: \texttt{chenwb@gmail.com}

\vspace{.1cm}
\textit{Email address}: S. Garcia-Balan: \texttt{sergiogb@yorku.ca}

\end{document}